\documentclass[a4paper,11pt]{article}
\usepackage{graphicx,amsmath,amsfonts,amssymb,amsthm,color,framed,dsfont,hyperref,accents, pgfplots}
\usepackage[T1]{fontenc} 
\usepackage[french,english]{babel}

\usepackage{geometry}
\pgfplotsset{width=10cm,compat=1.9}
\usetikzlibrary{external}
\usepackage{physics}
\usepackage{bbm}
\usepackage[font={small,it}]{caption}

\newcommand{\R}{\mathbb{R}}

\newcommand{\Int}{\mathsf{Int}}
\newcommand{\Cl}{\mathsf{Cl}}

\newcommand{\Prob}{\mathbb{P}}

\newcommand{\E}{\mathbb{E}}

\newcommand{\1}{\mathds{1}} 
\newcommand{\Wass}{\mathbb{W}}
\newcommand{\Law}{\mathcal{L}}

\newcommand{\e}{\mathrm{e}}

\newcommand{\Id}{\mathrm{Id}}

\newcommand{\eps}{\varepsilon}
\newcommand{\ds}{\displaystyle}

\newcommand{\cPc}{\mathcal{P}}
\newcommand{\cDc}{\mathcal{D}}

\newcommand{\cFc}{\mathcal{F}}

\newcommand{\MV}{{^{\scriptscriptstyle \text{MV}}}} 

\newcommand{\diff}{\sigma} 
\newcommand{\Diff}{\mathsf{\Sigma}} 
\newcommand{\drift}{b} 
\newcommand{\Drift}{\mathsf{B}} 
\newcommand{\X}{\mathfrak{X}} 

\newtheorem{thm}{Theorem}[section]
\newtheorem{corollary}[thm]{Corollary}

\newtheorem{lemma}[thm]{Lemma}

\newtheorem{prop}[thm]{Proposition}
\newtheorem{assumption}[thm]{Assumption}

\newtheorem{definition}[thm]{Definition}

\newtheorem{rem}[thm]{Remark}
\numberwithin{equation}{section}

\title{Freidlin-Wentzell type exit-time estimates for time-inhomogeneous diffusions and their applications
}

\date{\today}

\author{A.~Aleksian and S.~Villeneuve
\thanks{This work was supported by the Air Force Office of Scientific Research, Air Force Material Command, USAF (No. FA9550-19-7026).}
\\[5pt]
\small{Toulouse School of Economics}
}

\begin{document}

\maketitle

\begin{abstract}

This paper investigates the exit-time problem for time-inhomogeneous diffusion processes. The focus is on the small-noise behavior of the exit time from a bounded positively invariant domain. We demonstrate that, when the drift and diffusion terms are uniformly close to some time-independent functions, the exit time grows exponentially both in probability and in $L_1$ as a parameter that controls the noise tends to zero. We also characterize the exit position of the time-inhomogeneous process. Additionally, we investigate the impact of relaxing the uniform closeness condition on the exit-time behavior. As an application, we extend these results to the McKean-Vlasov process. Our findings improve upon existing results in the literature for the exit-time problem for this class of processes.

\end{abstract}

Vlasov process, exit time\par

 \par

\section{Introduction}
\label{s:Introduction}

The Freidlin-Wentzell theory emerged in the late 1960s, beginning with several works by the two authors (most notably \cite{VF70}), which were later compiled in their seminal book \cite{FW}. The problem originated from the study of small random perturbations in dynamical systems of the form  
\begin{equation} 
\label{eq:def_det_flow}  
    \dot{\phi}_t = \drift(\phi_t), \quad \phi_0 = x;  
\end{equation}  
where $\drift \in C_1(\R^d; \R^d)$ is a vector field defining the flow, and $x \in \R^d$ is the initial point.  

In particular, Freidlin-Wentzell theory focused on perturbations in the form of Brownian motion with a small diffusion parameter. Thus, they examined the family of processes $Z = Z(x, \eps)$, parameterized by the starting point $x \in \R^d$ and a parameter $\eps$ that controls the noise. Consider the following definition:  
\begin{equation}  
\label{eq:def_process_Z}  
        \dd{Z_t} = \sqrt{\eps} \diff(Z_t) \dd{W_t} + \drift(Z_t) \dd{t}, \quad Z_0 = x \text{ a.s.},  
\end{equation}  
where $\sigma \in C_1(\R^d; \R^{d \times d})$ is a matrix-valued function that introduces spatially dependent noise. In what follows, we will always assume that $\drift$ and $\diff$ are continuous Lipschitz functions. Under this condition, the existence and uniqueness results for \eqref{eq:def_process_Z} are standard. Additionally, every object will be defined on a sufficiently rich filtered probability space $(\Omega, \cFc, (\cFc_t)_{t \geq 0}, \Prob)$ and $W$ will be a $\cFc_t$-Brownian motion. By $(\cFc^W_t)_{t \geq 0}$, we denote the Brownian filtration associated to $W$.

Several problems were considered for \eqref{eq:def_process_Z} by M.~Freidlin and A.~Wentzell, among them, the question of the first exit time from a positively invariant for the flow \eqref{eq:def_det_flow} set $\cDc$. Let $\cDc \subset \R^d$ be some open, bounded domain. We recall the standard assumptions of the Freidlin-Wentzell theory.

\begin{assumption}
\label{assu:1}
    There exists a unique stable equilibrium point for the flow induced by $\drift$ inside $\cDc$, denoted by $a \in \cDc$. That is, $\phi_0 \in \cDc \Rightarrow  \{\phi_t\}_{t \geq 0} \subset \cDc$ and $\ds \lim_{t \to \infty} \phi_t = a$, where $\phi$ is defined in \eqref{eq:def_det_flow}. 
\end{assumption}

The first assumption states that, indeed, the set $\cDc$ is a positively invariant set and with a unique attractor inside it. In particular, it means that $\drift(a) = 0$. The second assumption states that the vector field $\drift$ points strictly inside the set $\cDc$ on its boundary $\partial \cDc$. 

\begin{assumption}
\label{assu:2}
    $\langle b(x), n(x) \rangle < 0$ for all $x \in \partial \cDc$, where $n(x)$ is the unit outward normal at $x \in \partial \cDc$.
\end{assumption}

The next assumption concerns the behavior of $\drift$ around the attractor $a$. 

\begin{assumption}
\label{assu:3}
    All eigenvalues of the Jacobian matrix  $(\frac{\partial \drift_i}{\partial x_j}(a))_{i, j}$ are negative.
\end{assumption}

In particular, Assumption \ref{assu:3} means that the field $\drift$ exhibits very strong attractive behavior around $a$ with an additional term. In particular, there exists $L > 0$ such that $\langle \drift(x), x - a \rangle \leq -L|x - a|^2$. The next assumption introduces some restrictions on the diffusion term. 

\begin{assumption}
\label{assu:4}
    There exists $\lambda > 0$ such that for any $x \in \cDc$, we have $ \frac{1}{\lambda} < \|\diff(x)\| < \lambda$. Moreover, for any $x \in \cDc$, $(\diff\diff^*) (x)$ is a non-singular matrix that is uniformly positive definite on $\partial \cDc \cup \{a\}$.
\end{assumption}


Note that our assumptions can be generalized, as can be seen, for example, in Assumptions (A-1) to (A-4) of {\cite[Chapter 5]{DZ}}. Additionally, see Exercise 5.7.29 in the same reference for an explanation of how the assumptions used in this paper are special cases of those presented in the book.

As was said before, the Freidlin-Wentzell theory addresses the exit-time problem, i.e. the study of $\tau(Z) := \inf\{t \geq 0: Z_t \notin \cDc\}$--- the first time when the process $Z$ leaves the domain $\cDc$. 

The last ingredient that we need to formulate the exit-time result of Freidlin-Wentzell theory is the notion of quasi-potential and its height within $\cDc$. For any $t > 0$, $z \in \R^d$, and $f \in C[0, t]$ define the following functional. Let:
\begin{equation*}
    I_t^z(f) := \int_0^t \left\langle \dot{f}_s - \drift(f_s), (\diff\diff^*)^{-1} (f_s) \big(\dot{f}_s - \drift(f_s) \big) \right\rangle \dd{s},
\end{equation*}
if $\dot{f}$ exists and belongs to $L_2$ and $I_T^z(f) = + \infty$ otherwise. $I_T^z$ is called the rate function and comes from the large deviation principle for stochastic processes of the form \eqref{eq:def_process_Z}. For more information on the large deviation theory see \cite{DZ}. 

Let us now define $\ds U(z, y, t) := \inf_{f \in C[0, t] : f_t = y} I^{z}_t(f)$ and $\ds U(z, y) := \inf_{t \geq 0} U(z, y, t)$. 

The following function $Q(y) := U(a, y)$ is called the quasi-potential. The name quasi-potential comes from the fact that if $\drift = - \nabla V$ for some function (potential) $V: \R^d \to \R$ and $\diff$ is simply an identity matrix, then $Q(y) = V(y) - V(a)$ for all $y$ inside $\cDc$. Note that, under our assumptions, the quasi-potential $Q$ is continuous on $\partial \cDc$ and at point $a$ (see e.g. {\cite[p.236 Exercise 5.7.29]{DZ}}).

Finally, denote as
\begin{equation}
\label{eq:def:H}
    H := \inf_{y \in \partial \cDc} Q(y)
\end{equation} 
the height of the quasi-potential within the domain $\cDc$.

Under the Assumptions above, we have the following classical result of the Freidlin-Wentzell theory.
\begin{prop}[M.I.~Freidlin, A.D.~Wentzell]
\label{prop:FW_theory}
    For any compact $K \subset \cDc$, we have:
    \begin{enumerate}
        \item for any $\eta > 0$: $\ds \lim_{\eps \to 0}\inf_{x \in K} \Prob \left(\exp{\frac{2(H - \eta)}{\eps}} \leq \tau(Z) \leq \exp{\frac{2(H + \eta)}{\eps}} \right) = 1$,
        \item $\ds \lim_{\eps \to 0} \frac{\eps}{2} \log \E[\tau(Z)] = H$ uniformly for $x \in K$,
        \item for any closed $N \subset \partial \cDc$ such that $\ds \inf_{y \in N} Q(y) > H$, we have
        \begin{equation*}
            \lim_{\eps \to 0} \sup_{x \in K} \Prob \left(Z_{\tau(Z)} \in N \right) = 0.
        \end{equation*}
    \end{enumerate}
\end{prop}

This result shows that the exit time of the time-homogeneous Itô diffusion~\eqref{eq:def_process_Z} grows exponentially as $\eps \to 0$, both in probability and in $L_1$. Furthermore, the exit-position result indicates that, with high probability, the exit occurs near the subset of $\partial \cDc$ that minimizes the quasipotential $Q$. 

This classical result from Freidlin-Wentzell theory has since been refined in several ways. First, one may consider $\cDc$ as a characteristic domain for the flow, meaning it can include saddle points of $\drift$---something that is precluded by Assumption~\ref{assu:2} in our case. This extension was achieved, for example, in~\cite{Day90} by introducing a reflection on the boundary for~\eqref{eq:def_process_Z} and applying techniques similar to those in Freidlin and Wentzell theory to this modified process. 

Another question of significant interest in the field, due to its relevance for applications, is the derivation of sharper asymptotics or prefactors. Most results in this direction have been obtained for the reversible case (i.e., when $\drift$ is the gradient of some function, as in the Langevin equation) and without the diffusion matrix $\diff$. This focus arises because PDE techniques are used in these cases, rather than the large deviations principles employed in the Freidlin and Wentzell approach. The exit-time result in this context is known as the Eyring-Kramers formula:
\begin{equation*}
    \E[\tau(Z)] = C(1 + o_\eps(1)) \exp{\frac{2H}{\eps}}.
\end{equation*}
For the exact form of prefactors in the Eyring-Kramers formula and their derivation using a potential-theoretic approach, see~\cite{Bovier1, Bovier2} and references therein. See also~\cite{lelievre2022preprint} for recent advancements in this area. For results on exit times with prefactors in the nonreversible case (though still considering specific forms of $\drift$), see~\cite{Bouchet2016}. Proving the Eyring-Kramers formula under the general assumptions of Freidlin and Wentzell remains an open problem.

In this paper, we study the small-noise asymptotics of the exit time for a family of processes $Y = Y(\eps, x)$, parametrized by the noise parameter $\eps$, each represented in the form:
\begin{equation}
\label{eq:def:inhomogeneous_diff}
    \dd{Y_t} = \sqrt{\eps} \, \Diff(t, Y_t) \dd{W_t} + \Drift(t, Y_t) \dd{t}, \quad Y_0 = x \text{ a.s.},  
\end{equation}
where $\Diff$ and $\Drift$ may vary with $\eps$. Our goal is to identify conditions under which Freidlin-Wentzell-type exit-time estimates remain valid for such families of processes and to explore the implications of these conditions for specific choices of $\Diff$ and $\Drift$.

There are a lot of processes of this form. For example, McKean-Vlasov process, which is defined as a solution of the following stochastic differential equation
\begin{equation*}
    \dd{Y_t}  = \sqrt{\eps} \diff^\MV(Y_t, \mu_t) \dd{W_t} + \drift^\MV(Y_t, \mu_t)\dd{t},
\end{equation*}
where $\mu_t = \Law(Y_t)$ represents the law of the process $Y$. The exit time problem for this process was considered in several papers \cite{HIP2, EJP, AT24} in the case $\drift^\MV(x, \mu) = - \nabla V(x) - \int \nabla F(x - z) \mu(\dd{z})$, for some regular functions $V, F: \R^d \to \R$, also known as confinement and interaction potentials and $\diff^\MV$ being an identity matrix.

Another process of interest that can be represented in the form \eqref{eq:def:inhomogeneous_diff} is
\begin{equation*}
    \dd{Y_t} = \sqrt{\eps} \alpha(\rho_t(Y_t))\dd{W_t} - \nabla V(Y_t)\dd{t},
\end{equation*}
where, as before, $V: \R^d \to \R$ is some regular confinement potential, $\rho_t$ is the density of $Y_t$ and $\alpha: \R \to \R$ is some function that controls the the diffusion term (see \cite{BMV24swarm}).

\subsection{Main Results}

In the following, we will describe the precise assumptions on the process \eqref{eq:def:inhomogeneous_diff} under which the Freidlin-Wentzell type exit-time result is shown. We will only consider the processes such that their diffusion and drift terms a close enough to some functions that are independent of the time parameter, at least for some time. Define the following set of processes. 

\begin{definition}
    For any $\eps > 0$, $\kappa \geq 0$, and $x \in \R^d$, define the set $\X = \X(\eps, \kappa, x)$ \label{def:the_set_X} of all $\cFc^W_t$-adapted processes that are strong solutions of the following stochastic differential equation
    \begin{equation}
    \label{eq:def_main_process}
    \begin{cases}
        \dd{X_t} \!\!\! &= \sqrt{\eps} \, \Diff(t, X_t) \dd{W_t} + \Drift(t, X_t) \dd{t}, \\
        X_0 &= x \in  \R^d \text{ a.s.,}
    \end{cases}    
    \end{equation}
    for some $\Diff \in C_1(\R_+ \times \R^d; \R^{d \times d})$ and $\Drift \in C_1(\R_+ \times \R^d; \R^d)$ such that
    \begin{equation}
    \label{eq:def_X_eps_rho_x_K}
        \begin{aligned}
            &\sup_{t \geq 0} \sup_{y \in \R^d} \|\Diff(t, y) - \diff(y)\| \!\!\! &\leq \kappa, \\
            &\sup_{t \geq 0} \sup_{y \in \R^d} |\Drift(t, y) - \drift(y)| &\leq \kappa.
        \end{aligned}
    \end{equation}
\end{definition}

In other words, $\X$ is the set of stochastic processes such that they can be represented as solutions to the stochastic differential equation~\eqref{eq:def_main_process} with drift and diffusion terms that are close to $\diff$ and $\drift$ from Equation~\eqref{eq:def_process_Z}. Note that the set $\X$ is not empty for any $\eps, \kappa > 0$ and for any $x \in \R^d$, since the process $Z$, defined in \eqref{eq:def_process_Z} for the corresponding $\eps$, always belongs to $\X$. In what follows, we omit the parameters $(\eps, \kappa, x)$ when their meaning is clear from the context.

Let us consider the three parameters $\eps$, $\kappa$, and $x$ that are present in the definition of $\X(\eps, \kappa, x)$. The third parameter $x$ defines the starting point for all the processes $X$ that we consider. The parameter $\eps$ controls the noise of the processes that we consider. When $\eps$ is small, for fixed-independently-of-$\eps$ periods of time, the trajectory of $X$ tends to stay around the deterministic flow defined by $\dot{\phi}_t = \Drift(t, \phi_t)$, which could be shown using Itô calculus or more accurate exponential estimates provided by the large deviation principle (see e.g. \cite{DRST19}). The parameter $\kappa > 0$ controls how close the drift term $\Drift$ and the diffusion term $\Diff$ of the processes that we consider are to those of $Z$. Note that for $\kappa_1 < \kappa_2$, we have $\X(\eps, \kappa_1, x) \subset \X(\eps, \kappa_2, x)$. Moreover, the set $\X(\eps, 0, x)$ is also well-defined and consists of only one process, that is $Z$, the strong solution of \eqref{eq:def_process_Z} with corresponding $\eps$ and $x$. 

Any process $X$ that is a strong solution of \eqref{eq:def_main_process} is also a time-inhomogeneous strong Markov process. This property is thus preserved by the set $\X$, which we point out in the following remark. 

\begin{rem}
\label{rem:X_Markov_property}
    Take some $X \in \X(\eps, \kappa, x)$ and a stopping time $\tau$. For $t_0 > 0$ and $x^\prime \in \R^d$ define $X^\prime$ the strong solution of the SDE
    \begin{equation*}
    \begin{cases}
        \dd{X_t^\prime} \!\!\! &= \sqrt{\eps} \, \Diff(t_0 + t, X_t^\prime) \dd{W_t} + \Drift(t_0 + t, X_t^\prime) \dd{t}, \\
        X_0^\prime &= x^\prime \text{ a.s.}
    \end{cases}    
    \end{equation*}
    The functions $\Diff(t_0 + \cdot, \cdot)$ and $\Drift(t_0 + \cdot, \cdot)$ satisfy conditions \eqref{eq:def_X_eps_rho_x_K} and thus $X^\prime \in \X(\eps, \kappa, x^\prime)$. Moreover, by the strong Markov property of $X$, for any measurable function $f$, we have $\E [f(X^{\prime}_t)] = \E [f(X_{\tau + t})| X_{\tau} = x^\prime, \tau = t_0]$  for any $t \geq 0$.
\end{rem}

Finally, for a continuous $\cFc_t$-adapted process $X$, let us define the random time
\begin{equation}
\label{def:time_tau}
    \tau(X) := \inf\{t \geq 0: X_t \notin \cDc\}.
\end{equation}
Note that, under our assumptions, $\tau(X)$ is a stopping time (see, e.g. \cite[p.11]{karatzas2014brownian}). Now, we are ready to state the main result of the paper.

\begin{thm}
\label{thm:main_result}
    Fix some compact $K \subset \cDc$.  For any $\eta > 0$ there exists $\kappa > 0$ such that
    \begin{enumerate}
        \item $\ds \lim_{\eps \to 0} \inf_{x \in K}\inf_{X \in \X} \Prob \left(\exp{\frac{2(H - \eta)}{\eps}} \leq \tau(X) \leq \exp{\frac{2(H + \eta)}{\eps}} \right) = 1,$

        \item $\ds H - \eta \leq \lim_{\eps \to 0} \frac{\eps}{2} \log \inf_{x \in K} \inf_{X \in \X}\E [\tau(X)] \leq  \lim_{\eps \to 0} \frac{\eps}{2} \log \sup_{x \in K}\sup_{X \in \X}\E [\tau(X)] \leq H + \eta$,
       \item moreover, for any closed $N \subset \partial \cDc$ such that $\ds \inf_{y \in N} Q(y) > H + \eta$, we have
        \begin{equation*}
            \lim_{\eps \to 0} \sup_{x \in K} \sup_{X \in \X} \Prob \left(X_{\tau(X)} \in N \right) = 0.
        \end{equation*}
    \end{enumerate}
\end{thm}

The result above shows that the exponential growth of the exit time, both in probability and in $L_1$, also holds for a family of inhomogeneous diffusions as long as their drift and diffusion terms are close to those of \eqref{eq:def_process_Z}. The parameter $\eta$ determines the desired precision of the exponential growth rate. The parameter $\kappa$, which controls the closeness of the drift terms $\Drift$ and $\drift$, as well as the diffusion terms $\Diff$ and $\diff$, is then chosen as a function of $\eta$. Note that the choice of $\kappa$ may also depend on other objects fixed by our assumptions, such as the functions $\drift$, $\diff$, and the domain $\cDc$. In principle, an explicit formula describing the dependency of $\kappa$ on $\eta$ can be obtained, but it requires more precise estimates than those provided in this paper.

The inverse is also true. Namely, if the process $Y(\eps, x)$ is given by \eqref{eq:def:inhomogeneous_diff} with some $\Drift$ and $\Diff$ such that there exist $\drift$, $\diff$, $a$, and $\cDc$ that satisfy Assumptions~\ref{assu:1}--\ref{assu:4} and such that $Y(\eps, x) \in \X(\eps, \kappa, x)$, for any $\eps$ and for {\it small enough} $\kappa > 0$, then there exists $\eta$ such that the exit time $\tau(Y)$ exhibits the exponential growth of the form of Theorem~\ref{thm:main_result}. Therefore, the Freidlin-Wentzell type exit time estimate holds even in the case when $\Drift$ and $\Diff$ does not tend to $\drift$ and $\diff$ with $\eps \to 0$. Although, the fact that $\kappa$ shall be "small enough" is essential. Its exact value can be found using more accurate estimates on the dependency of $\kappa$ on $\eta$ in Theorem~\ref{thm:main_result}.

If in fact the drift $\Drift$ and diffusion $\Diff$ tend to some independent-of-time limits $\drift$ and $\diff$ as $\eps \to 0$, then Theorem~\ref{thm:main_result} boils down to an equivalent to the Freidlin-Wentzell theory result. Consider the following simple corollary.

\begin{corollary}
\label{cor:main_result_uniform}
Fix some compact $K \subset \cDc$. Let $Y(\eps, x)$ be a family of $\cFc_t$-adapted processes such that for any $\kappa > 0$ and for any $\eps > 0$ small enough we have
\begin{equation}
\label{eq:cor_main_condition}
    Y(\eps, x) \in \X(\eps, \kappa, x) \text{, for any } x \in K.
\end{equation}
Then we have the following three results: 
\begin{enumerate}
    \item for any $\delta > 0$: $\ds \lim_{\eps \to 0}\inf_{x \in K} \Prob \left(\exp{\frac{2(H - \delta)}{\eps}} \leq \tau(Y) \leq \exp{\frac{2(H + \delta)}{\eps}} \right) = 1$,
    \item $\ds \lim_{\eps \to 0} \frac{\eps}{2} \log \E[\tau(Y)] = H$ uniformly for $x \in K$,
    \item moreover, for any closed $N \subset \partial \cDc$ such that $\ds \inf_{y \in N} Q(y) > H$, we have
    \begin{equation*}
        \lim_{\eps \to 0} \sup_{x \in K} \Prob \left(Y_{\tau(Y)} \in N \right) = 0.
    \end{equation*}
\end{enumerate}
\end{corollary}

In this paper, the term 'small enough' indicates that the upper bound of a parameter depends on the parameters specified earlier. For instance, in the case above, the equivalent formulation would be: `` $<$...$>$ for any $\kappa > 0$ there exists $\eps(\kappa) > 0$ such that for any positive $\eps < \eps(\kappa)$ we have $<$...$>$''.

The first result of the corollary can be achieved by simply taking for any $\delta > 0$ the parameter $\kappa$ given by Theorem~\ref{thm:main_result}. Then, we observe that, due to our condition on $Y$, regardless of the value of $\kappa$, in the limit $\eps \to 0$, we have $Y(\eps, x) \in \X(\eps, \kappa, x)$, thus giving
\begin{equation*}
\begin{aligned}
    &\lim_{\eps \to 0} \Prob \left(\exp{\frac{2(H - \delta)}{\eps}} \leq \tau(Y) \leq \exp{\frac{2(H + \delta)}{\eps}} \right) \\
    &\geq \lim_{\eps \to 0} \inf_{X \in \X} \Prob \left(\exp{\frac{2(H - \delta)}{\eps}} \leq \tau(X) \leq \exp{\frac{2(H + \delta)}{\eps}} \right) = 1.
\end{aligned}
\end{equation*}
This estimate, of course, holds uniformly in $x \in K$. The same holds for the exit position:
\begin{equation*}
    \lim_{\eps \to 0} \sup_{x \in K} \Prob \left(Y_{\tau(Y)} \in N \right) \leq \lim_{\eps \to 0} \sup_{x \in K} \sup_{X \in \X} \Prob \left(X_{\tau(X)} \in N \right) = 0. 
\end{equation*}

For the $L_1$-asymptotics, one can proceed similarly and note that, if $\ds \lim_{\eps \to 0} \frac{\eps}{2} \log \E[\tau(Y)]$ exists, then it should satisfy both
\begin{equation*}
   \lim_{\eps \to 0} \frac{\eps}{2} \log \E[\tau(Y)] \geq \lim_{\eps \to 0} \frac{\eps}{2} \log \inf_{X \in \X}\E [\tau(X)] \geq H - \delta,  
\end{equation*}
and 
\begin{equation*}
   \lim_{\eps \to 0} \frac{\eps}{2} \log \E[\tau(Y)] \leq \lim_{\eps \to 0} \frac{\eps}{2} \log \sup_{X \in \X}\E [\tau(X)] \leq H + \delta.  
\end{equation*}
These two bounds on the limit $\ds \lim_{\eps \to 0} \frac{\eps}{2} \log \E[\tau(Y)]$, which does not depend on $\kappa$ or $\delta$, hold for any $\delta > 0$. Moreover, it holds uniformly on $x \in K$, thus giving us the second result of the corollary.

One may argue that the conditions used for $Y$ in the corollary above are too strong. Indeed, we suppose that, in the small-noise regime, the processes $Y(\eps, x)$ lose their time inhomogeneous properties for all $t \geq 0$. This, in general, is not true for many processes of interest, as will be seen below for the McKean-Vlasov process. In the following remark, we note one obvious way to relax the condition \eqref{eq:cor_main_condition}. 

\begin{rem}
The result of Corollary~\ref{cor:main_result_uniform} still holds if we change the condition \eqref{eq:cor_main_condition} to the following. For any $\kappa > 0$, for any $\eps > 0$ small enough, and for any $x \in K$ there exists a process $X \in \X(\eps, \kappa, x)$ that is indistinguishable from $Y$ until time $\tau(X)$. Specifically,
\begin{equation}
\label{eq:cor_rabdom_time_condition}
    \Prob (\forall t \leq \tau(X): Y_t = X_t) = 1.
\end{equation}    
\end{rem}

The proof of the Freidlin-Wentzell exit time result in this case is obvious; one needs to slightly modify the argument above using the fact that, though $Y$ may not belong to $\X$, there exists $X \in \X(\eps, \kappa, x)$ such that $\tau(X) = \tau(Y)$ almost sure. 

Another possible way to relax the condition \eqref{eq:cor_main_condition} is to consider processes $Y(\eps, x)$ that belong to $\X(\eps, \kappa, x)$ only until some deterministic time that is big enough for the exit event to occur. Consider the following corollary.

\begin{corollary}
\label{cor:Kramers_type_law}
    Fix some compact $K \subset \cDc$. Let $Y(\eps, x)$ be a family of $\cFc_t$-adapted processes. Suppose that there exists $\delta^\prime > 0$ such that for any $\kappa > 0$, for any $\eps > 0$ small enough, and for any $x \in K$ there exists a process $X \in \X(\eps, \kappa, x)$ that is indistinguishable from $Y$ until time $\exp{\frac{2(H + \delta^\prime)}{\eps}}$. Specifically,  
    \begin{equation*}
        \Prob \left(\forall t \leq \exp{\frac{2(H + \delta^\prime)}{\eps}}: Y_t = X_t \right) = 1.
    \end{equation*} 
    
    Then for any $\delta > 0$, we have:
    \begin{enumerate}
        \item $\ds \lim_{\eps \to 0}\inf_{x \in K} \Prob \left(\exp{\frac{2(H - \delta)}{\eps}} \leq \tau(Y) \leq \exp{\frac{2(H + \delta)}{\eps}} \right) = 1$,
        \item $\ds \lim_{\eps \to 0} \frac{\eps}{2} \log \inf_{x \in K}\E[\tau(Y)] \geq H$,
        \item moreover, for any closed $N \subset \partial \cDc$ such that $\ds \inf_{y \in N} Q(y) > H$, we have
        \begin{equation*}
            \lim_{\eps \to 0} \sup_{x \in K} \Prob \left(Y_{\tau(Y)} \in N \right) = 0.
        \end{equation*}
    \end{enumerate}
\end{corollary}

The proof of this result is direct. Take, without loss of generality, $\delta < \delta^\prime$. Then, since $Y$ is indistinguishable from some $X$ until time $\exp{\frac{2(H + \delta^\prime)}{\eps}}$, we have:
\begin{equation*}
    \Prob\left(\e^{\frac{2(H - \delta)}{\eps}} \leq \tau(Y) \leq \e^{\frac{2(H + \delta)}{\eps}} \right) \geq \Prob\left(\e^{\frac{2(H - \delta)}{\eps}} \leq \tau(X) \leq \e^{\frac{2(H + \delta)}{\eps}} \right).
\end{equation*}
The exit-position result can be achieved using the estimate on the exit time $\tau(Y)$. Indeed:
\begin{equation*}
    \Prob \left(Y_{\tau(Y)} \in N \right) \leq \Prob \left(X_{\tau(X)} \in N \right) + \Prob \left(\tau(Y) > \e^{\frac{2(H + \delta^\prime)}{\eps}} \right),
\end{equation*}
which tends to zero by Theorem~\ref{thm:main_result} and the estimate in $\Prob$ on the exit time $\tau(Y)$. 

For the lower bound in $L_1$, that is the second result of the corollary, we can simply use the Markov's inequality to get for any $\delta > 0$:
\begin{equation*}
    \Prob\left(\e^{\frac{2(H - \delta)}{\eps}} \leq \tau(Y) \right) \e^{\frac{2(H - \delta)}{\eps}} \leq \E[\tau(Y)].
\end{equation*}
Since $\Prob\left(\e^{\frac{2(H - \delta)}{\eps}} \leq \tau(Y) \right) \xrightarrow[\eps \to 0]{} 1$, that gives us 
\begin{equation*}
     \lim_{\eps \to 0} \frac{\eps}{2} \log \inf_{x \in K}\E[\tau(Y)] \geq H - \delta,
\end{equation*}
for any $\delta > 0$ thus proving the second statement of the corollary.

Note that, unlike before, in Corollary~\ref{cor:Kramers_type_law} we get only a lower bound on the exponential rate of growth in $L_1$. In fact, under the conditions on the process $Y$, we can not hope for more. Indeed, take for example the process 
\begin{equation*}
    Y_t(\eps, x) := Z_t(\eps, x) \1 \left\{t \leq \e^{\frac{2(H + 1)}{\eps}} \right\} + a \1 \left\{t > \e^{\frac{2(H + 1)}{\eps}} \right\},
\end{equation*}
where $Z$ is defined in \eqref{eq:def_process_Z} (we recall that $Z(\eps, x) \in \X(\eps, 0, x)$). The following probability
\begin{equation*}
    \Prob \left(\tau(Y) > \e^{\frac{2(H + 1)}{\eps}} \right) \geq \Prob \left(\tau(Z) > \e^{\frac{2(H + 1)}{\eps}} \right)
\end{equation*}
is positive for any $\eps > 0$, thus making $\E[\tau(Y)] = \infty$, since the process $Y$ is forever trapped in $a$ after time $\exp{\frac{2(H + 1)}{\eps}}$.

Although this example may appear artificial, it illustrates an important concept: the behavior of $Y$ after the time $\exp{\frac{2(H + 1)}{\eps}}$ influences the expectation of the exit time $\tau(Y)$, but not its exponential asymptotics in $\Prob$.



\subsection{Applications for the McKean-Vlasov process}

In this section, we present the exit-time result for the McKean-Vlasov process that we establish using Theorem~\ref{thm:main_result}.

Let us define the McKean-Vlasov process that is considered in this section. For each $\eps > 0$ and $x \in \R^d$, consider the following equation:
\begin{equation}
\label{eq:def_McKean-Vlasov}
\begin{cases}
    \dd{Y_t} \!\!\! &= \sqrt{\eps} \diff^\MV(Y_t, \mu_t) \dd{W_t} + \drift^\MV(Y_t, \mu_t)\dd{t}, \\
    \mu_t &= \Law(Y_t), \\
    Y_0 & = x \text{ a.s.}
\end{cases}
\end{equation}
where $\diff^\MV \in C_1(\R^d \times \mathcal{P}_2(\R^d); \R^{d \times d})$ and $\drift^\MV  \in C_1(\R^d \times \mathcal{P}_2(\R^d); \R^d)$.

Let us introduce some assumptions on the drift and diffusion terms. Since this paper does not address the questions of existence and uniqueness, we proceed under the following condition:

\begin{assumption}
\label{assu:MV:existence}
    The functions $\diff^\MV$ and $\drift^\MV$ are such that for any $\eps > 0$ and $x \in \R$ there exists a unique strong solution to \eqref{eq:def_McKean-Vlasov}. Furthermore, there exist $M > 0$ and $p \geq 2$ such that for some compact $K^\prime \subset \R^d$
    \begin{equation*}
        \sup_{0 < \eps < 1}\sup_{x \in K^\prime}\sup_{t \geq 0} \E|Y_t|^{p} < M.
    \end{equation*}
\end{assumption}

We refer to \cite{AT24, HIP2} for some possible conditions that can be taken to ensure Assumption~\ref{assu:MV:existence} in the gradient case where $\drift^\MV(x, \mu) = - \nabla V(x) - \nabla F*\mu(x)$ and $\diff^\MV(x, \mu) = \Id$.

The process \eqref{eq:def_McKean-Vlasov} can also be viewed in the form \eqref{eq:def:inhomogeneous_diff} by associating $\Diff(t, x) = \diff^\MV(x, \mu_t)$ and $\Drift(t, x) = \drift^\MV(x, \mu_t)$. However, the nature of nonhomogeneity of $Y$ is important. Theorem~\ref{thm:main_result} addresses the exit-time problem from a positively invariant set of some kind, with a unique attractor inside it. Since we consider the exit time in the small noise regime, it is natural to expect the law $\mu_t$ to be close to the Dirac delta function centered at the attractor. Consider the following assumption.

\begin{assumption}
\label{assu:MV:FWstuff}
    There exists $a \in \R^d$ such that $\drift = \drift^\MV(\cdot, \delta_a)$ and $\diff = \diff^\MV(\cdot, \delta_a)$ satisfy Assumptions~\ref{assu:1}--\ref{assu:4} for some open $\cDc \subset K$.
\end{assumption}

Assumption~\ref{assu:MV:FWstuff} not only introduces some restrictions on $\drift^\MV$ and $\diff^\MV$, but also brings us back in the framework of Theorem~\ref{thm:main_result} by associating $\drift = \drift^\MV(\cdot, \delta_a)$ and $\diff = \diff^\MV(\cdot, \delta_a)$. Note that in the following it is these $\drift$ and $\diff$ that are understood in Equation~\eqref{eq:def_main_process} and Definition~\ref{def:the_set_X}. 

However, in order to obtain the exit-time result in the case of McKean-Vlasov diffusion process, we need three additional assumptions.

\begin{assumption}
\label{assu:MV:control1}
    There exists an increasing positive function $C: \R \times \R \to \R$ such that for any $\mu, \nu \in \cPc_2(\R^d)$, we have: 
    \begin{equation*}
    \begin{aligned}
        \sup_{x \in \R^d}|\drift^\MV(x, \mu) - \drift^\MV(x, \nu)| \leq C\left(\int |z|^p\mu(\dd{z}), \int |z|^p\nu(\dd{z}) \right) \Wass_2(\mu, \nu), \\ 
        \sup_{x \in \R^d}|\diff^\MV(x, \mu) - \diff^\MV(x, \nu)| \leq C\left(\int |z|^p\mu(\dd{z}), \int |z|^p\nu(\dd{z}) \right) \Wass_2(\mu, \nu),
    \end{aligned}
    \end{equation*}
    where $p$ is defined in Assumption~\ref{assu:MV:existence}. 
\end{assumption}

Note that Assumptions~\ref{assu:MV:existence} and \ref{assu:MV:control1} give us that for $\widetilde{C}_1 := C(M, |a|^p)$, for any $t \geq 0$, we have: 
\begin{equation}
\label{eq:drift_control}
\begin{aligned}
    \sup_{x \in \R^d} |\drift^\MV(x, \mu_t) - \drift(x)| \leq \widetilde{C}_1 \Wass_2(\mu_t, \delta_a), \\
    \sup_{x \in \R^d} |\diff^\MV(x, \mu_t) - \diff(x)| \leq \widetilde{C}_1 \Wass_2(\mu_t, \delta_a), \\
\end{aligned}
\end{equation}

\begin{assumption}
\label{assu:MV:control2}
    There exists a radius $R > 0$ such that for any $\mu, \nu \in \cPc_2(\R^d)$, we have:
    \begin{equation*}
        \sup_{x \in B_R(a)}\|\nabla_{\!x}\drift^\MV(x, \mu) - \nabla_{\!x}\drift^\MV(x, \nu)\| \leq C\left(\int |z|^p\mu(\dd{z}), \int |z|^p\nu(\dd{z}) \right) \Wass_2(\mu, \nu).
    \end{equation*}
\end{assumption}

Similarly to Equation~\eqref{eq:drift_control}, we get for any $t \geq 0$:
\begin{equation*}
    \sup_{x \in B_{R}(a)} \|\nabla_{\!x}\drift^\MV(x, \mu_t) - \nabla \drift(x)\| \leq \widetilde{C}_1 \Wass_2(\mu_t, \delta_a).
\end{equation*}

Note that by Assumption~\ref{assu:3}, the Jacobian matrix $\nabla_{\!x} \drift^\MV(a, \delta_a) = \nabla \drift(a)$ is negative-definite, that is there exists a positive constant $\widetilde{K} > 0$ such that:
\begin{equation*}
    \nabla_{\!x} \drift^\MV(a, \delta_a) = \nabla \drift(a) \preceq - \widetilde{K} \, \Id.
\end{equation*}
Note also that $\nabla_{\!x} \drift^\MV(x, \delta_a)$ is assumed to be continuous and, therefore, this convexity property is preserved in any small neighborhood of $a$. Specifically, there exists a function $\widetilde{K}(r) \in (0, \widetilde{K}]$ such that $\widetilde{K}(r) \xrightarrow[r \to 0]{} \widetilde{K}$ and for any $r>0$ small enough, and for any $x \in B_r(a)$ we have:
\begin{equation}
\label{eq:nabla_drift_control}
    \nabla_{\!x} \drift^\MV(x, \delta_a) = \nabla \drift(x) \preceq - \widetilde{K}(r) \, \Id.
\end{equation}
Without loss of generality, let us say that $R$ is chosen to be small enough such that the property~\eqref{eq:nabla_drift_control} holds for any $r \in [0, R]$.

The last assumption we have to make concerns the control of the drift term subjected to a small perturbation of the measure around the attractor $a$. We recall that $\drift^\MV(a, \delta_a) = \drift(a) = 0$. Consider the following assumption.

\begin{assumption}
\label{assu:MV:control3}
    For any $\mu \in \cPc_2(\R^d)$, we have $|\drift^\MV(a, \mu)| < \widetilde{K} \Wass_2(\mu, \delta_a)$.
\end{assumption}

In particular, the assumption above means that there exists $\widetilde{K}_1\in (0, \widetilde{K})$ such that:
\begin{equation}
\label{eq:drift_at_a_control}
     |\drift^\MV(a, \mu)| \leq \widetilde{K}_1 \Wass_2(\mu, \delta_a).
\end{equation}
We will make use of this notion in what follows.

The general approach used to resolve the exit-time problem for McKean-Vlasov process consists in utilizing the tool presented in Corollary~\ref{cor:Kramers_type_law}. In order to do so, we need to control in some sense the variation of the drift and diffusion terms. In light of Equation~\eqref{eq:drift_control}, it is in fact enough to control the $\Wass_2$-distance between the law of the process $\mu_t$ and $\delta_a$ for a sufficiently long time. Namely, it is sufficient to prove the following lemma.

\begin{lemma}
\label{lm:McKean-Vlasov_law_control}
    There exists $\delta^\prime$ such that for any $\kappa > 0$ and for any $\eps > 0$ small enough, we have:
    \begin{equation}
    \label{eq:aux:Wass_leq_rho}
        \Wass_2(\mu_t, \delta_a) \leq \kappa/\widetilde{C}_1,
    \end{equation}
    for any $t \leq \exp{\frac{2(H + \delta^\prime)}{\eps}}$. 
\end{lemma}

By Equation~\eqref{eq:drift_control}, the result of the lemma gives us
\begin{equation*}
\begin{aligned}
    \sup_{t \leq \exp{\frac{2(H + \delta^\prime)}{\eps}}} \sup_{x \in \R^d} |\drift^\MV(x, \mu_t) - \drift(x)| \leq \kappa, \\
    \sup_{t \leq \exp{\frac{2(H + \delta^\prime)}{\eps}}} \sup_{x \in \R^d} |\diff^\MV(x, \mu_t) - \diff(x)| \leq \kappa,
\end{aligned}
\end{equation*}
which brings us back to the conditions of Corollary~\ref{cor:Kramers_type_law}. 

The intuition behind this statement is clear. As a result of $\cDc$ being a basin of attraction for $a$ for the noiseless flow $\dot{\varphi}(t) = \drift^\MV(\varphi(s), \delta_a)$, for small values of the noise parameter $\eps$, we expect $Y_t$ to be found around $a$ with high probability. However, the proof of such a statement is not obvious, since the window of time $\left[0, \exp{\frac{2(H + \delta^\prime)}{\eps}} \right]$ also expands with a decrease of $\eps$. Corollary~\ref{cor:Kramers_type_law} gives us the following result in the case of McKean-Vlasov process. 

\begin{thm}
\label{thm:McKean_exit_time}
    Let Assumptions~\ref{assu:MV:existence}--\ref{assu:MV:control3} hold and let $Y$ be the unique strong solution of \eqref{eq:def_McKean-Vlasov} with $x = a$. Then for any $\delta > 0$, we have:
    \begin{enumerate}
        \item $\ds \lim_{\eps \to 0}\inf_{x \in K} \Prob \left(\exp{\frac{2(H - \delta)}{\eps}} \leq \tau(Y) \leq \exp{\frac{2(H + \delta)}{\eps}} \right) = 1$,
        \item $\ds \lim_{\eps \to 0} \frac{\eps}{2} \log \inf_{x \in K}\E[\tau(Y)] \geq H$,
        \item moreover, for any closed $N \subset \partial \cDc$ such that $\ds \inf_{y \in N} Q(y) > H$, we have
        \begin{equation*}
            \lim_{\eps \to 0} \sup_{x \in K} \Prob \left(Y_{\tau(Y)} \in N \right) = 0.
        \end{equation*}
    \end{enumerate}    
\end{thm}

This result extends the findings in \cite{AT24}, which considered the specific case $\drift^\MV(x, \mu_t) = - \nabla V(x) - \nabla F*\mu_t(x)$ and $\diff^\MV(x, \mu_t)$ as the identity matrix, by providing analogous results for more general drift and diffusion terms. Furthermore, Corollaries~\ref{cor:main_result_uniform} and \ref{cor:Kramers_type_law} clarify why, in \cite{AT24}, the authors could only show the exit-time asymptotics in $\Prob$. In contrast, the earlier study \cite{HIP2}, which assumes convexity of both $V$ and $F$, additionally established:
\begin{equation*}
     \lim_{\eps \to 0} \frac{\eps}{2} \log \inf_{x \in K}\E[\tau(Y)] = H.
\end{equation*}
The convexity of $V$ and $F$ ensures control of the law $\mu_t$ \textit{for all} $t \geq 0$, thereby aligning with the framework of Corollary~\ref{cor:main_result_uniform} rather than Corollary~\ref{cor:Kramers_type_law}. At the same time, in \cite{AT24}, potentials $V$ and $F$ are not assumed to be convex, thus making the uniform in time control of the law impossible in general. This shows that both results in \cite{HIP2} and \cite{AT24} can be viewed as particular cases of the broader results presented in this paper.

\section{Proof of the main result}

Our proof of Theorem~\ref{thm:main_result} follows the logic of the Freidlin-Wentzell theory for the time-homogeneous case, as presented in \cite{FW, DZ}. However, adapting their results to our time-inhomogeneous setting is challenging and requires careful consideration. We start the proof by providing some auxiliary lemmas. 

The first result demonstrates that, with high probability, the distance between $X$ and $Z$ remains within a predetermined interval over a given in advance time. Additionally, by reducing $\kappa$, we can make sure that the probability approaches 1 exponentially fast, when $\eps$ tends to 0. This lemma is not used directly in the proof of Theorem~\ref{thm:main_result}, yet it is essential for other auxiliary results.

\begin{lemma}
\label{lm:|X_t-Z_t|>delta}
For any $\kappa > 0$, for any $T, \delta > 0$, there exists a constant $C > 0$ such that for any $\eps \leq 1$, we have
    \begin{equation*}
        \frac{\eps}{2} \log \sup_{x \in \cDc} \sup_{X \in \X} \Prob (\sup_{t \in [0, T]} |X_t - Z_t| > \delta) \leq C + \log(\frac{\kappa^2}{\kappa^2 + C^2\delta^2}) =: -\varsigma(\kappa, T, \delta)
    \end{equation*}
\end{lemma}


The next lemma provides a uniform in $x \in \cDc$ and in $X \in \X$ lower bound for the probability that the stopping time $\tau(X)$ is less than or equal to a given in advance positive time $T_1$. This lower bound is given in terms of an exponentially decaying function of $\eps$.

\begin{lemma}
\label{lm:P_tau_D_lower_bound}
    For any $\eta > 0$ there exist $\kappa > 0$ and a positive time $T_1$, such that for any $0 < \eps < 1$ we have
    \begin{equation*}
        \inf_{x \in \cDc} \inf_{X \in \X} \Prob (\tau(X) \leq T_1) \geq \exp{-\frac{2(H + \eta/2)}{\eps}},        
    \end{equation*}
\end{lemma}

The following lemma shows that the probability that the process, starting from a small sphere around $a$, hits $\partial \cDc$ before coming closer to $a$ and hitting an even smaller sphere is exponentially small, uniformly in $X \in \X$. Define 
\begin{equation}
\label{eq:def:tau_prime}
    \tau_{r}^\prime(X) := \inf\{t \geq 0: X_t \in B_{r}(a)\cup\partial\cDc\}
\end{equation} 
and consider

\begin{lemma}
\label{lm:tau_prime}
    For any $\eta > 0$ and for any $\rho > 0$ small enough there exists $\kappa > 0$ such that
    \begin{equation*}
        \limsup_{\eps \to 0} \frac{\eps}{2} \log \sup_{x \in S_{2 \rho}(a)} \sup_{X \in \X} \Prob (X_{\tau_{\rho}^\prime(X)} \in \partial \cDc) \leq -H + \frac{\eta}{2}.
    \end{equation*}
\end{lemma}

Lemma~\ref{lm:tau_prime} is a particular case of the following result of Lemma~\ref{lm:tau_prime_N}. We present them separately to increase readability and since Lemma~\ref{lm:tau_prime} is used for the proof of the exit time result while Lemma~\ref{lm:tau_prime_N} is used for the exit location. Nevertheless, they will be proved together.

\begin{lemma}
\label{lm:tau_prime_N}
    For any closed $N \subset \partial \cDc$, for any $\delta > 0$ and for any $\rho > 0$ small enough there exists $\kappa > 0$ such that
    \begin{equation*}
        \limsup_{\eps \to 0} \frac{\eps}{2} \log \sup_{x \in S_{2 \rho}(a)} \sup_{X \in \X} \Prob (X_{\tau_{\rho}^\prime(X)} \in N) \leq - \inf_{y \in N} Q(y) + \delta.
    \end{equation*}
\end{lemma}

The following result is similar to Lemmas~\ref{lm:tau_prime} and \ref{lm:tau_prime_N}. The main difference here is that the process $X$ starts within a fixed compact subset of $\cDc$, rather than in a small neighborhood around $a$. Nevertheless, since $\cDc$ is the domain of attraction, the process $X$ will, with high probability, still approach $a$ before exiting $\cDc$. The trade-off is that, in Lemmas~\ref{lm:tau_prime} and \ref{lm:tau_prime_N}, the probability decreases exponentially fast as $\eps \to 0$, whereas in the following result, we only assert that this probability tends to zero, without specifying the rate.

\begin{lemma}
\label{lm:tau_prime=tau}
For any $K \subset \cDc$ and for any $\rho > 0$ small enough there exists $\kappa > 0$, such that:
    \begin{equation*}
        \sup_{x \in K}\sup_{X \in \X} \Prob(\tau^\prime_{\rho}(X) = \tau(X)) = \sup_{x \in K}\sup_{X \in \X} \Prob(X_{\tau^\prime_{\rho}(X)} \in \partial \cDc) \xrightarrow[\eps \to 0]{} 0.
    \end{equation*}
\end{lemma}

The last lemma establishes that the process $X$ remains close to its starting point $x \in \cDc$ over a short time interval. Specifically, for any predetermined distance $r > 0$, there exists a small positive time $T_2$ such that the probability that $X$ moves more than $r$ away from its starting point $x$ within time $T_2$ tends to zero exponentially fast when $\eps \to 0$. Furthermore, this convergence is uniform over $x \in \cDc$. Consider the following lemma.

\begin{lemma}
\label{lm:X_small_dev_from_x}
    For any $\eta, \rho > 0$  there exist $\kappa > 0$ and a positive time $T_2$ such that 
    \begin{equation*}
        \limsup_{\eps \to 0} \frac{\eps}{2} \log \sup_{x \in \cDc} \sup_{X \in \X} \Prob \left(\sup_{t \leq T_2} |X_t - x| \geq \rho/2 \right) \leq -H - 2\eta.
    \end{equation*}
\end{lemma}

Now we are ready to present the proof of Theorem~\ref{thm:main_result}.

\subsection{Proof of Theorem~\ref{thm:main_result}}

   Before proceeding with the proof, we clarify the choice of certain parameters used throughout. We recall that additionally to the objects defined by the assumptions of the paper, we are also given a compact $K \subset \cDc$ and a closed set $N \in \partial \cDc$ by the statement of the proposition. Fix any $\eta> 0$ and choose $\delta > 0$ such that
   \begin{equation}
       \label{eq:aux:choice_of_delta}
       \inf_{y \in N} Q(y) - \delta > H + \eta + \delta.
   \end{equation}
   That is possible by the assumption on the set $N$ provided in the proposition. Fix $\rho > 0$ small enough as required by Lemmas~\ref{lm:tau_prime}, \ref{lm:tau_prime_N}, and \ref{lm:tau_prime=tau}.
   
   After that, choose $\kappa$ as the minimum of the $\kappa$'s defined in Lemmas~\ref{lm:P_tau_D_lower_bound}--\ref{lm:X_small_dev_from_x}. Since each lemma remains valid when $\kappa$ is reduced, this choice ensures that all the lemmas hold simultaneously. With this being said, we are now ready to start the proof of the proposition.
    
    \textit{Upper bound in $\Prob$.} Define $\displaystyle q := \inf_{x \in \cDc} \inf_{X \in \X} \Prob (\tau(X) \leq T_1)$, where $T_1$ is the positive time defined in Lemma~\ref{lm:P_tau_D_lower_bound}. According to this lemma, we have
    \begin{equation}
    \label{aux:eq:q_lower_bound}
        q \geq \exp{-\frac{2(H + \frac{\eta}{2})}{\eps}}.
    \end{equation}
    By Remark~\ref{rem:X_Markov_property}, we can use the Markov property of the processes $X \in \X$ and get:
    \begin{equation*}
    \begin{aligned}
        \Prob (\tau(X) > (k + 1)T_1) &= \Big[ 1 - \Prob(\tau(X) \leq (k + 1)T_1 \;|\; \tau(X) > k T_1)\Big] \Prob (\tau(X) > k T_1) \\
        & \leq (1 - q) \, \Prob (\tau(X) > k T_1).
    \end{aligned}
    \end{equation*}
    Therefore, by induction on $k$, we get the following bound:
    \begin{equation*}
        \sup_{x \in \cDc} \sup_{X \in \X} \Prob (\tau(X) > k T_1) \leq (1 - q)^k.
    \end{equation*}
    For any positive random variable $\xi$, we can estimate its expectation by a simple formula $\ds \E[\xi] \leq T_1 \left(1 + \sum_{k = 1}^\infty \Prob(\xi > kT_1) \right)$. In the case of $\tau(X)$, that gives us:
    \begin{equation*}
        \sup_{x \in \cDc} \sup_{X \in \X} \E[\tau(X)] \leq T_1 \Big(1 + \sum_{k = 1}^\infty \sup_{x \in \cDc} \sup_{X \in \X} \Prob (\tau(X) > k T_1) \Big) \leq T_1 \sum_{k = 0}^\infty (1 - q)^k = \frac{T_1}{q}.
    \end{equation*}
    Using this upper bound and the inequality \eqref{aux:eq:q_lower_bound}, we finally get the following upper bound
    \begin{equation}
    \label{eq:aux:E_tau_control}
         \sup_{x \in \cDc} \sup_{X \in \X} \E[\tau(X)] \leq T_1 \exp{\frac{2(H + \frac{\eta}{2})}{\eps}},
    \end{equation}
    which, by Markov's inequality, gives us:
    \begin{equation*}
        \sup_{x \in \cDc} \sup_{X \in \X}\Prob \left(\tau(X) > \exp{\frac{2(H + \eta)}{\eps}} \right) \leq \e^{-\frac{2(H + \eta)}{\eps}} \sup_{x \in \cDc} \sup_{X \in \X} \E [\tau(X)] \leq T_1 \e^{- \eta/\eps},
    \end{equation*}
    which tends to 0 when $\eps \to 0$, which proves the upper bound.

    \textit{Lower bound in $\Prob$.} To prove the lower bound, let us define the following stopping times (see Fig.~\ref{fig:theta_i}). Let $\theta_0 = 0$ and define for $i = 1, 2, \dots$:
    \begin{equation*}
        \begin{aligned}
            \theta_{2i - 1} &:= \inf\{t \geq \theta_{2i - 2}: X_t \in B_{\rho}(a) \cup \partial \cDc\}, \\
            \theta_{2i} &:= 
            \begin{cases}
                \inf \{t \geq \theta_{2i - 1}: X_t \in S_{2 \rho}(a)\}, \text{ if } X_{\theta_{2i - 1}} \in B_{\rho}(a); \\
                \infty, \text{ if } X_{\theta_{2i - 1}} \in \partial \cDc.
            \end{cases}
        \end{aligned}
    \end{equation*}

    \begin{figure}[t]
    \centering
    \includegraphics{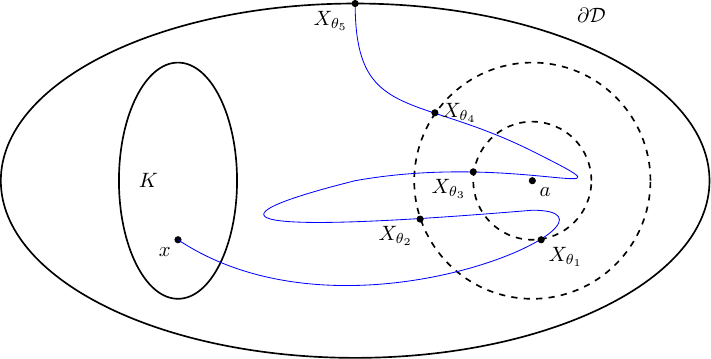}

    \caption{An illustration of the stopping times $\theta_i$.}
    \label{fig:theta_i}
    \end{figure}

Note that $\{X_{\theta_i}\}_{i \geq 0}$ is a sequence of random variables that take their values on $B_{\rho}(a) \cup S_{2 \rho}(a) \cup \partial\cDc$ and $\Prob(\exists i: \theta_{2i - 1} = \tau(X)) = 1$. Moreover, using Lemmas~\ref{lm:tau_prime}--\ref{lm:X_small_dev_from_x}, we can estimate the time-length of the transitions. First, by Lemma~\ref{lm:tau_prime=tau}, we have
\begin{equation}
\label{eq:aux:First}
    \sup_{x \in K}\sup_{X \in \X} \Prob (\theta_1 = \tau(X)) \xrightarrow[\eps \to 0]{} 0.
\end{equation}

Second, for $\theta_{2i - 1}$, with $i \geq 2$, we shall use a different approach. Use Remark~\ref{rem:X_Markov_property}, that is the strong Markov property of the processes $X \in \X$ and Lemma~\ref{lm:tau_prime}, and get
\begin{equation}
\label{eq:aux:Second}
    \sup_{x \in \cDc} \sup_{X \in \X} \Prob (\theta_{2i - 1} = \tau(X)) \leq \sup_{x \in S_{2\rho}(a)} \sup_{X \in \X} \Prob (X_{\tau^\prime_{\rho}(X)} \in \partial \cDc) \leq \e^{\frac{-2(H - \frac{\eta}{2})}{\eps}},
\end{equation}
for $\eps$ small enough. Third, note that for any $i \geq 1$, by Lemma \ref{lm:X_small_dev_from_x}, we get:
\begin{equation}
\label{eq:aux:Third}
    \sup_{x \in \cDc} \sup_{X \in \X}  \Prob (\theta_{2i} - \theta_{2i - 1} \leq T_2) \leq  \sup_{x \in B_{\rho}(a)} \sup_{X \in \X}  \Prob (\sup_{t \leq T_2} |X_t - x| \geq \rho/2) \leq \e^{\frac{-2(H + 2\eta)}{\eps}},
\end{equation}
for $\eps$ small enough.
    
Note that if the event $\{\tau(X) \leq kT_2 \}$ holds true, then either $\{\theta_{2i - 1} = \tau(X)\}$ occurs for some $1 \leq i \leq k$ or one of the intervals $\theta_{2i} - \theta_{2i - 1}$ is less or equal to $T_2$. In particular, we have the following upper bound for any $x \in K$ and for any $X \in \X = \X(\eps, \kappa, x)$:
\begin{equation*}
    \Prob (\tau(X) \leq kT_2)  \leq \sum_{i = 1}^{k} \Prob (\theta_{2i - 1} = \tau(X)) + \Prob \left(\min_{i \leq k} (\theta_{2i} - \theta_{2i - 1}) \leq T_2\right). 
\end{equation*}
Using Remark~\ref{rem:X_Markov_property} and \eqref{eq:aux:Second}, we get that $\ds \Prob \left(\min_{i \leq k} (\theta_{2i} - \theta_{2i - 1}) \leq T_2\right) \leq k \e^{-\frac{2(H + 2\eta)}{\eps}}$, thus giving us:

\begin{equation*}    
    \Prob (\tau(X) \leq kT_2) \leq \Prob (\theta_{1} = \tau(X)) + 2k \exp{-\frac{2(H - \eta/2)}{\eps}}.
\end{equation*}

We can apply this inequality for $k = \left\lfloor\exp{\frac{2(H - \eta)}{\eps}}/T_2 \right\rfloor + 1$ and get
\begin{equation*}
\begin{aligned}
    \sup_{X \in \X} \Prob \left( \tau(X) \leq \exp{\frac{2(H - \eta)}{\eps}} \right) & \leq \sup_{X \in \X} \Prob (\theta_1 = \tau(X)) + 4\e^{-\eta/\eps}/T_2,
\end{aligned}
\end{equation*}
which tends to zero by \eqref{eq:aux:First}.

\textit{Upper and lower bounds in $L_1$.} The second result of the proposition is a simple consequence of the previous proof. Indeed, for the fixed above $\eta$ and $\kappa$, Equation~\eqref{eq:aux:E_tau_control} gives:
\begin{equation*}
    \frac{\eps}{2}\log \sup_{x \in K} \sup_{X \in \X} \E[\tau(X)] \leq \frac{\eps}{2}\log(T_1) + \left(H + \eta \right) \xrightarrow[\eps \to 0]{} H + \eta.
\end{equation*}

For the lower bound, we will use Markov's inequality and get:
\begin{equation*}
    \inf_{x \in K}\inf_{X \in \X} \E [\tau(X)] \geq \exp{\frac{2(H - \eta)}{\eps}} \inf_{x \in K}\inf_{X \in \X} \Prob\left(\exp{\frac{2(H - \eta)}{\eps}} \leq \tau(X)\right),
\end{equation*}
which gives us
\begin{equation*}
\begin{aligned}
    \frac{\eps}{2} \log \inf_{x \in K}\inf_{X \in \X} \E [\tau(X)] &\geq (H - \eta) + \frac{\eps}{2} \log \inf_{x \in K}\inf_{X \in \X} \Prob\left(\exp{\frac{2(H - \eta)}{\eps}} \leq \tau(X)\right) \\
    &\xrightarrow[\eps \to 0]{} (H - \eta).    
\end{aligned}
\end{equation*}

\textit{Exit location.} The proof of this statement is based on the estimate already obtained for the exit time $\tau(X)$ and Lemma~\ref{lm:tau_prime_N} that suggests that the probability to exit through the set $N$ is exponentially smaller than that from $\partial \cDc$ in general. As we know, due to the estimate on the exit time, the probability that $\tau(X) > \exp{\frac{2(H + \eta)}{\eps}}$ tends to zero as $\eps \to 0$. Thus, consider:
\begin{equation*}
    \Prob(X_{\tau(X)} \in N) \leq \Prob\left(X_{\tau(X)} \in N, \tau(X) \leq \e^{\frac{2(H + \eta)}{\eps}}\right) + \Prob\left(\tau(X) > \e^{\frac{2(H + \eta)}{\eps}}\right).
\end{equation*}

Similarly to the upper bound, for any integer $k$, we can separate the first event the following way.
\begin{equation*}
    \{X_{\tau(X)} \in N, \tau(X) \leq \e^{\frac{2(H + \eta)}{\eps}}\} \subset \bigcup_{i = 1}^{k}\{X_{\theta_{2i - 1}} \in N\} \cup \{\theta_{2k} \leq \tau(X) \leq \e^{\frac{2(H +\eta)}{\eps}}\}.
\end{equation*}

The probability of the event on the right can be expressed as:
\begin{equation*}
\begin{aligned}
    \Prob(\theta_{2k} \leq \e^{\frac{2(H +\eta)}{\eps}}) &\leq \Prob\left(\min_{i \leq k} \left(\theta_{2i } - \theta_{2i - 1} \right) \leq \e^{\frac{2(H +\eta)}{\eps}}/k \right) \\
    &\leq \sum_{i = 1}^k \Prob\left(\theta_{2i } - \theta_{2i - 1} \leq \e^{\frac{2(H +\eta)}{\eps}}/k \right).     
\end{aligned}
\end{equation*}
Indeed, if all the intervals of the form $\theta_{2i} - \theta_{2i - 1}$ were greater than $\e^{\frac{2(H +\eta)}{\eps}}/k$ then $\theta_{2k}$ would be also greater than $\e^{\frac{2(H +\eta)}{\eps}}$. Remark~\ref{rem:X_Markov_property} allows us to use the Markov property and conclude that
\begin{equation*}
\begin{aligned}
    \Prob \left(\theta_{2k} < \e^{\frac{2(H +\eta)}{\eps}} \right) &\leq k \sup_{x \in B_{\rho}(a)} \sup_{X \in \X} \Prob(\theta_2 < \e^{\frac{2(H +\eta)}{\eps}}/k)\\
    &\leq k \sup_{x \in B_{\rho}(a)} \sup_{X \in \X} \Prob(\sup_{t \leq \e^{\frac{2(H +\eta)}{\eps}}/k} |X_t - x| > \rho/2).    
\end{aligned}
\end{equation*}

Similarly, using Remark~\ref{rem:X_Markov_property}, for $k > 1$, we get:
\begin{equation*}
\begin{aligned}
    \sup_{x \in K} \sup_{X \in \X} \Prob(\bigcup_{i = 1}^{k}\{X_{\theta_{2i - 1}} \in N\} ) &\leq \sup_{x \in K}\sup_{X \in \X} \Prob(X_{\tau^\prime_\rho(X)} \in \partial \cDc) \\
    & \quad + (k - 1) \sup_{x \in S_{2\rho}(a)}\sup_{X \in \X} \Prob(X_{\tau^\prime_\rho(X)} \in N).    
\end{aligned}
\end{equation*}
By Lemmas~\ref{lm:tau_prime_N} and \ref{lm:X_small_dev_from_x}, and by our choice of $\delta$, we get the following estimate:
\begin{equation*}
\begin{aligned}
    \sup_{x \in K} \sup_{X \in \X} \Prob(\bigcup_{i = 1}^{k}\{X_{\theta_{2i - 1}} \in N\} ) &\leq \alpha(\eps) + (k -1) \exp{-\frac{2(\inf_{y \in N} Q(y) - \delta)}{\eps}} \\
    &\leq \alpha(\eps) + (k -1) \exp{-\frac{2(H + \eta + \delta)}{\eps}},    
\end{aligned}
\end{equation*}
where $\alpha(\eps) \xrightarrow[\eps \to 0]{} 0$.

Combining the  estimates, we finally get:
\begin{equation*}
\begin{aligned}
    \Prob(X_{\tau(X)} \in N) &\leq \alpha(\eps) + k  \sup_{x \in B_{\rho}(a)} \sup_{X \in \X} \Prob \left(\sup_{t \leq \e^{\frac{2(H +\eta)}{\eps}}/k} |X_t - x| > \rho/2 \right) \\   
    & + (k -1) \exp{-\frac{2(H + \eta + \delta)}{\eps}}.
\end{aligned}
\end{equation*}
Take $k = \left\lfloor \frac{1}{T_2} \exp{\frac{2(H + \eta)}{\eps}} \right\rfloor$. This makes sure that $\exp{\frac{2(H + \eta)}{2}}/k \leq T_2$, which allows us to use Lemma~\ref{lm:X_small_dev_from_x}. We thus get:
\begin{equation*}
    \Prob(X_{\tau(X)} \in N) \leq \alpha(\eps) + \frac{3}{T_2} \e^{-2\delta/\eps} \xrightarrow[\eps \to 0]{} 0, 
\end{equation*}
which ends the proof.

\subsection{Proofs of auxiliary lemmas}

Before proving the lemmas, we provide some remarks. Note that for any $\rho > 0$, we can define $\cDc^e_{\rho} \supset \cDc$ \label{def:D_enlargement} as an enlargement of $\cDc$ that satisfies Assumptions~\ref{assu:1}, \ref{assu:2}, and \ref{assu:4}, and ensures $d(\cDc^e_{\rho}, \cDc) \geq \rho$. Additionally, let us denote $H^e_{\rho} := \inf\{V(a, z): z \in \cDc^e_{\rho}\}$ as the height of the quasipotential inside the set $\cDc^e_{\rho}$. Finally, note that the results concerning the process $Z$ used in the following computations are presented in the Appendix on page~\pageref{s:Appendix}. With these remarks in place, we are now ready to prove the auxiliary results.

\begin{proof}[Proof of Lemma~\ref{lm:|X_t-Z_t|>delta}]
    The following result {\cite[Lemma 5.6.18]{DZ}} states that, for any Itô process $D$ of the form
    \begin{equation*}
        \dd{D_t} = \sqrt{\eps} \diff_t \dd{W_t} + \drift_t \dd{t},
    \end{equation*}
    where $\drift$ is a progressively measurable processes and $\diff$ is progressively measurable and bounded, such that, for some positive constants $C_1$, $C_2$, and $\rho$, we have 
    \begin{equation*}
        \begin{aligned}
            |\diff_t| & \leq C_1 (\rho^2 + |D_t|^2)^{1/2}, \\
            |\drift_t| & \leq C_2 (\rho^2 + |D_t|^2)^{1/2}
        \end{aligned}
    \end{equation*}

    Then for any $\delta > 0$, positive time $T$ and for any $\eps < 1$, we have
    \begin{equation*}
        \frac{\eps}{2} \log \Prob (\sup_{t \leq T} |D_t| \geq \delta) \leq K + \log\left(\frac{\rho^2 + |D_0|^2}{\rho^2 + \delta^2}\right),
    \end{equation*}
    where $K = 2C_2 + C_1^2(2 + d)$, which, importantly, does not depend on $D$ itself.

    In our case, $D_t = X_t - Z_t$ and, therefore
    \begin{equation*}
        \dd{D_t} = \sqrt{\eps} \left(\Diff(t, X_t) - \diff(Z_t)\right) \dd{W_t} + \left(\Drift(t, X_t) - \drift(Z_t)\right) \dd{t}.
    \end{equation*}
    
    The conditions on $\diff$ and $\drift$ are also satisfied. Consider e.g. $\diff_t$
    \begin{equation*}
        \left\|\Diff(t, X_t) - \diff(Z_t)\right\| \leq \left\|\Diff(t, X_t) - \diff(X_t) + \diff(X_t) - \diff(Z_t)\right\| \leq \kappa + C|X_t - Z_t|,
    \end{equation*}
    where we get the last inequality by the definition of $\X$ and the Lipshitz condition on $\diff$. Therefore,
        \begin{equation*}
        \left\|\Diff(t, X_t) - \diff(Z_t)\right\| \leq \sqrt{2}C\left((\kappa/C)^2 + |X_t - Z_t|^2 \right)^{1/2} = \sqrt{2}C\left((\kappa/C)^2 + |D_t|^2 \right)^{1/2}.
    \end{equation*}
    The same holds for $\drift_t = \Drift(t, X_t) - \drift(Z_t)$. Now we just have to use the result for $D$ that is independent of the choice of $x \in \cDc$ and $X \in X$ and since $D_0 = 0$ a.s., we get:
    \begin{equation*}
        \frac{\eps}{2} \log \sup_{x \in \cDc} \sup_{X \in \X} \Prob(\sup_{t \leq T} |X_t - Z_t| \geq \delta) \leq K + \log\left(\frac{\kappa^2}{\kappa^2 + C^2\delta^2}\right),
    \end{equation*}
    which shows the result of the lemma by renaming the constants.
\end{proof}

\begin{proof}[Proof of Lemma~\ref{lm:P_tau_D_lower_bound}]
    Recall that $\cDc^e_{\rho} \supset \cDc$ is an enlargement of $\cDc$ that satisfies Assumptions~\ref{assu:1}, \ref{assu:2}, and \ref{assu:4}, and $d(\cDc^e_{\rho}, \cDc) \geq \rho$. $H^e_{\rho} := \inf\{V(a, z): z \in \cDc^e_{\rho}\}$ denotes the height of the  quasipotential inside the set $\cDc^e_{\rho}$. Since the quasipotential $Q$ is continuous, we can fix $\rho > 0$ as a small enough constant such that $H^e_\rho < H + \frac{\eta}{8}$.
    
    For any continuous $\cFc_t$-adapted process $X$, define the stopping time $\tau^e_{\rho}(X) := \inf\{t \geq 0: X_t \notin \cDc^{e}_{\rho}\}$. By Lemma~\ref{lm:tau(Z)<T_lower_bound}, for any $\rho > 0$, there exists $T_1 > 0$ such that we have
    \begin{equation}
    \label{eq:aux:tau_Z_bound}
        \inf_{x \in \cDc}\Prob \!\left(\tau^e_{\rho}(Z) \leq T_1 \right) \geq \exp{-\frac{2(H^e_{\rho} + \frac{\eta}{8})}{\eps}} \geq \exp{-\frac{2(H + \frac{\eta}{4})}{\eps}}.
    \end{equation}

    As the next step, note that for each $x \in \cDc$ and for each $X \in \X$, the following inclusion of events holds $\displaystyle \{\tau^e_{\rho}(Z) \leq T_1, \sup_{t \in [0, T_1]}|X_t - Z_t| \leq \rho/2\} \subseteq \{\tau(X) \leq T_1\}$. Consider the following computations:
    \begin{equation*}
    \begin{aligned}
        \inf_{x \in \cDc} \inf_{X \in \X}  \Prob (\tau(X) \leq T_1) &\geq \inf_{x \in \cDc} \inf_{X \in \X} \Prob \left(\tau^e_{\rho}(Z) \leq T_1, \sup_{t \in [0, T_1]} |X_t - Z_t| \leq \rho/2 \right) \\
        & \geq  \inf_{x \in \cDc} \Prob (\tau^e_{\rho}(Z) \leq T_1) - \sup_{x \in \cDc} \sup_{X \in \X}  \Prob \left(\sup_{t \in [0, T_1]}|X_t - Z_t| > \rho/2 \right) \\
        & \geq \exp{-\frac{2(H + \frac{\eta}{4})}{\eps}} - \exp{\frac{-2 \varsigma(\kappa, T_1, \frac{\rho}{2})}{\eps}},
    \end{aligned}
    \end{equation*}
    
    where the last inequality is obtained from \eqref{eq:aux:tau_Z_bound} and Lemma~\ref{lm:|X_t-Z_t|>delta}. Therefore, for $\eps < 1$, we have:
    \begin{equation*}
        \inf_{x \in \cDc} \inf_{X \in \X}  \Prob (\tau(X) \leq T_1) \geq \e^{-\frac{2(H + \frac{\eta}{4})}{\eps}} \left(1 - \exp{-2(\varsigma(\kappa, T_1, \frac{\rho}{2}) - H - \frac{\eta}{4})} \right).
    \end{equation*}

    Choose $\kappa > 0$ to be small enough such that
    \begin{equation*}
        \exp{-2 \frac{\eta}{4}} \leq \left(1 - \exp{-2(\varsigma(\kappa, T_1, \frac{\rho}{2}) - H - \frac{\eta}{4})} \right),
    \end{equation*}
    which is possible since $\varsigma(\kappa, T_1, \frac{\rho}{2}) \xrightarrow[\kappa \to 0]{} + \infty$. Finally, that gives us
    \begin{equation*}
        \inf_{x \in \cDc} \inf_{X \in \X}  \Prob (\tau(X) \leq T_1) \geq \exp{-\frac{2(H + \eta/4)}{\eps} - \frac{2 \eta}{4}} \geq \exp{-\frac{2(H + \eta/2)}{\eps}},
    \end{equation*}
    since $\eps < 1$. That completes the proof, since $\eps \to 0$.
\end{proof}

\begin{proof}[Proof of Lemmas~\ref{lm:tau_prime} and \ref{lm:tau_prime_N}]

    As was pointed out above, Lemma~\ref{lm:tau_prime} is a particular case of Lemma~\ref{lm:tau_prime_N}. therefore, we will only prove the second one. Fix closed $N \subset \partial \cDc$ and any $\delta > 0$. Consider the set $\overline{N}_\rho := \{x \in \R^d: d(x, N) \leq \rho\}$. Fix $\rho > 0$ small enough so that, first, $S_{2\rho}(a) \subset \cDc$ and, second, $\ds \inf_{y \in \overline{N}_\rho} Q(y) \geq \inf_{y \in N} Q(y) - \frac{\delta}{4}$, which is possible due to continuity of $Q$ on $\partial \cDc$. Moreover, since $Q$ is also continuous around $a$ and $Q(a) = 0$, we have $\ds \lim_{\rho \to 0} \sup_{z \in S_{2\rho}(a)} Q(z) = 0$. Decrease $\rho$ so that $\ds \sup_{z \in S_{2\rho}(a)} Q(z) \leq \frac{\delta}{4}$. The parameter $\kappa$ will be chosen later. We will prove this lemma in two steps.
    
    {\it Step 1.} First, let us prove the following. There exists $T^\prime > 0$ such that for any $\eps > 0$ small enough, we have
    \begin{equation*}
        \sup_{x \in S_{2 \rho}(a)} \sup_{X \in \X} \Prob (\tau^\prime_{\rho}(X) > T^\prime) \leq \e^{-2 H/\eps}.
    \end{equation*}
    
    Indeed, by Lemma~\ref{lm:tau_prim_rho(Z)>t}, there exists a time $T^\prime > 0$ such that for $\gamma^1_\rho := \inf\{t \geq 0: Z_t \in B_{\rho/4}(a) \cup \cDc^e_{3\rho/4}\}$ we have
    \begin{equation}
    \label{eq:aux:tau_prime_T}
        \limsup_{\eps \to 0} \frac{\eps}{2} \log \sup_{x \in S_{2 \rho}(a)} \Prob (\gamma^1_\rho > T^\prime) < -H - 1.
    \end{equation}

    Note that the following inclusion of the events holds: $\{\tau^\prime_{\rho}(X) > T^\prime, \sup_{t \leq T^\prime} |X_t - Z_t| \leq \rho/2\} \subseteq \{\gamma^1_{\rho} > T^\prime \}$. Then we have
    \begin{equation*}
            \Prob (\tau^\prime_{\rho}(X) > T^\prime) \leq \Prob \Big(\sup_{t \leq T^\prime} |X_t - Z_t| > \rho/2 \Big) + \Prob (\gamma^{1}_\rho > T^\prime).
    \end{equation*}

    By Lemma~\ref{lm:|X_t-Z_t|>delta} and Equation \eqref{eq:aux:tau_prime_T}, for any $\eps > 0$ small enough, we have the following upper bound:
    \begin{equation*}
    \begin{aligned}
        \sup_{x \in S_{\rho}(a)} &\sup_{X \in \X} \Prob (\tau^\prime_{\rho}(X) > T^\prime) \leq \exp{-2 \varsigma(\kappa, T^\prime, \frac{\rho}{2})/\eps} + \exp{- 2(H + 1)/\eps} \\
        & \leq \exp{-\frac{2H}{\eps}} \left(\exp{-\frac{2(\varsigma(\kappa, T^\prime, \frac{\rho}{2}) - H)}{\eps}} + \exp{-\frac{2}{\eps}} \right).
    \end{aligned}
    \end{equation*}
    Therefore, we can finally choose $\kappa$ to be small enough such that $\left(\varsigma(\kappa, T^\prime, \frac{\rho}{2}) - H \right) > 1$ and thus:
    \begin{equation*}
          \sup_{x \in S_{2 \rho}(a)} \sup_{X \in \X} \Prob (\tau^\prime_{\rho}(X) > T^\prime) \leq \exp{-\frac{2 H}{\eps}}.
    \end{equation*}

    {\it Step 2.}  To finish the proof, we introduce $\Phi:= \{f \in C[0, T^\prime]: f(0) \in S_{2\rho}(a),  \text{ and } \exists t \leq T^\prime \text{ such that } f(t) \in \overline{N}_\rho \}$. According to Lemma~\ref{lm:Z_in_Phi}, and our choice of $\rho$, we have:
    \begin{equation*}
        \lim_{\eps \to 0} \frac{\eps}{2} \log \sup_{x \in S_{2 \rho}(a)} \Prob(Z \in \Phi) \leq - \inf_{y \in N} Q(y) + \frac{\delta}{2}.
    \end{equation*}
    We use the decomposition of the events to get that, if $X_{\tau^\prime_\rho(X)} \in N$, then either $\tau^\prime_\rho(X) > T^\prime$ or $\ds \sup_{s \leq T^\prime} |Z_s - X_s| > \rho/2$, or, in none of these cases, $\Prob(Z \in \Phi)$. Consider the following inequality: 
    \begin{equation*}
    \begin{aligned}
        \Prob (X_{\tau^\prime_\rho(X)} \in N) &\leq \Prob (\tau_\rho^\prime(X) > T^\prime) + \Prob(Z \in \Phi) +  \Prob (\sup_{s \leq T_\rho^\prime}|Z_s - X_s| > \rho/2).
    \end{aligned}
    \end{equation*}

    By Lemma~\ref{lm:|X_t-Z_t|>delta}, we have
    \begin{equation*}
        \Prob \left(\sup_{s \leq T_\rho^\prime}|Z_s - X_s| > \rho/2 \right) \leq \exp{-2 \varsigma(\kappa, T^\prime, \frac{\rho}{2})/\eps}.
    \end{equation*}
    Choose $\kappa > 0$ small enough so that $\ds \varsigma(\kappa, T^\prime, \frac{\rho}{2}) > \inf_{y \in N} Q(y) - \frac{\delta}{2}$. This gives us the following upper bound for small enough $\eps$:
    \begin{equation*}
         \sup_{x \in S_{2 \rho}(a)} \sup_{X \in \X} \Prob (X_{\tau_{\rho}^\prime(X)} \in N) \leq 3 \exp{\frac{2(-\inf_{y \in N} Q(y) + \delta/2)}{\eps}}, 
    \end{equation*}
    thus proving the lemma.

\end{proof}

\begin{proof}[Proof of Lemma~\ref{lm:tau_prime=tau}]
    Note that, by our assumptions, there exists a uniform in $x \in K$ upper bound for the time of convergence of the dynamical system $\dot{\phi}_t = b(\phi_t)$ inside $B_{\rho/4}(a)$. Moreover, the trajectories of $\phi$ do not leave from $\cDc$. Rigorously, there exists a positive time $T$, such that for any $x \in K$, we have $\phi_{T} 
    \in B_{\rho/4}(a)$ and also
    \begin{equation*}
        \bigcup_{x \in K} \{\phi_t\}_{t \leq T} \subset \cDc.
    \end{equation*}
    Let us fix $\delta < \inf_{x \in K} d(\partial \cDc, \{\phi_t\}_{t \leq T}) \wedge \rho$, which is possible since $\inf_{x \in K} d(\partial \cDc, \{\phi_t\}_{t \leq T}) > 0$. By Lemma~\ref{lm:Z_phi_control}, we have
    \begin{equation*}
        \sup_{x \in K} \Prob(\sup_{t \leq T}|Z_t - \phi_t| > \delta/4) \xrightarrow[\eps \to 0]{} 0.
    \end{equation*}
    Additionally, by Lemma~\ref{lm:|X_t-Z_t|>delta}, we can choose $\kappa$ small enough such that
    \begin{equation*}
        \sup_{x \in K} \sup_{X \in \X} \Prob(\sup_{t \leq T}|Z_t - X_t| > \delta/4) \xrightarrow[\eps \to 0]{} 0.
    \end{equation*}
    In order to finish the proof, note that
    \begin{equation*}
        \sup_{x \in K}\sup_{X \in \X} \Prob(X_{\tau^\prime_{\rho}(X)} \in \partial \cDc) \leq \sup_{x \in K}\sup_{X \in \X} \Prob(\sup_{t \leq T}|X_t - \phi_t| > \delta/2).
    \end{equation*}
    Indeed, otherwise, the process $X$ would hit $B_{\rho}(a)$ before leaving from $\cDc$. We finish the proof by noting that
    \begin{equation*}
    \begin{aligned}
        \sup_{x \in K}\sup_{X \in \X} \Prob(\sup_{t \leq T}|X_t - \phi_t| > \delta/2) &\leq \sup_{x \in K} \Prob(\sup_{t \leq T}|Z_t - \phi_t| > \delta/4) \\
        & \quad+ \sup_{x \in K} \sup_{X \in \X} \Prob(\sup_{t \leq T}|Z_t - X_t| > \delta/4),        
    \end{aligned}
    \end{equation*}
    which tends to zero as $\eps \to 0$, as was shown above.
\end{proof}

\begin{proof}[Proof of Lemma~\ref{lm:X_small_dev_from_x}]
    Fix some $\rho, \eta > 0$. By Lemma~\ref{lm:|Z_t-x|>r}, there exists $T_2 > 0$ such that
    \begin{equation}
    \label{eq:aux:|Z_t-x|}
        \limsup_{\eps \to 0} \frac{\eps}{2} \log \sup_{x \in \cDc} \Prob \left(\sup_{t \leq T_2} |Z_t - x| \geq \rho/4 \right) \leq -(H + 2\eta).
    \end{equation}

    Note that we have the following inequalities:
    \begin{equation*}
    \begin{aligned}
        \sup_{x \in \cDc} \sup_{X \in \X}  \Prob \left(\sup_{t \leq T_2} |X_t - x| \geq \frac{\rho}{2} \right) &\leq \sup_{x \in \cDc } \Prob \left(\sup_{t \leq T_2} |Z_t - x| \geq \rho/4 \right) \\ 
        & \quad + \sup_{x \in \cDc} \sup_{x \in \X} \Prob \left(\sup_{t \leq T_2} |X_t - Z_t| \geq \rho/4 \right).      
    \end{aligned}
    \end{equation*}
    By Lemma~\ref{lm:|X_t-Z_t|>delta} and Equation \eqref{eq:aux:|Z_t-x|}, for $\eps < 1$, the last expression is bounded by:
    \begin{equation*}
        \sup_{x \in \cDc} \sup_{X \in \X}  \Prob \left(\sup_{t \leq T_2} |X_t - x| \geq \frac{\rho}{2} \right) \leq \e^{-2(H + 2\eta)/\eps}\left(1 + \e^{-2(\varsigma(\kappa, T_2, \frac{\rho}{4}) - H -2\eta)} \right).
    \end{equation*}
    Choose $\kappa > 0$ to be small enough such that $\exp{-2(\varsigma(\kappa, T_2, \frac{\rho}{4}) - H -2\eta)} < 1$. Then the following upper bound holds:
    \begin{equation*}
         \sup_{x \in \cDc} \sup_{X \in \X}  \Prob \left(\sup_{t \leq T_2} |X_t - x| \geq \frac{\rho}{2} \right) \leq 2\exp{-2(H + 2\eta)/\eps},
    \end{equation*}
    for $\eps$ small enough, which proves the lemma, after taking $\frac{\eps}{2} \log$ from both sides of the inequality.
\end{proof}

\section{Exit time for McKean-Vlasov diffusion}

Recall that Theorem~\ref{thm:McKean_exit_time} follows directly from Lemma~\ref{lm:McKean-Vlasov_law_control} that is control of the law of the process $Y$ in the form:
    \begin{equation*}
        \Wass_2(\mu_t, \delta_a) \leq \kappa/\widetilde{C}_1.
    \end{equation*}
We can prove the control of the law in the form above not only for the McKean-Vlasov but for any inhomogeneous process attracted strong enough to $a$. 

Let $M, R, \widetilde{K}_1 > 0$ be as defined in Assumptions~\ref{assu:MV:existence}--\ref{assu:MV:control3}. Consider the following subset of $\X$.
\begin{definition}
\label{def:X_tilde}
    For any $\eps > 0$, $\kappa > 0$, and $x \in \R^d$, let $\widetilde{\X}(\eps, \kappa, x)$ be the set of $\cFc^W_t$-adapted processes that belong to $\X(\eps, \kappa, x)$ and satisfy three additional conditions:
    \begin{enumerate}
        \item $\ds \sup_{x \in \R^d} \sup_{X \in \widetilde{X}} \sup_{t \geq 0}\E|X_t - a|^4 < M,$ 
        \item $\ds \sup_{t \geq 0} \sup_{x \in B_R(a)} \|\nabla_{\! x} \Drift(t, x) - \nabla \drift(x)\| \leq \kappa$,
        \item $\ds \sup_{t \geq 0}|\Drift(t, a)| \leq \widetilde{K}_1 \frac{\kappa}{\widetilde{C}_1} $.
    \end{enumerate}
\end{definition}

The conditions of the definition are clearly inspired by Assumption~\ref{assu:MV:existence} and Equations~\eqref{eq:nabla_drift_control} and \eqref{eq:drift_at_a_control}. In fact, it is easy to check that if we define $\ds S_\kappa := \inf\{t \geq 0: \Wass_2(\mu_t, \delta_a) \geq \kappa/\widetilde{C}_1\}$, then the stopped process $\{Y_{t \wedge S_\kappa} \}_{t\geq 0}$ belongs to $\widetilde{\X}(\eps, \kappa, a)$ for any $\kappa, \eps > 0$. This makes the following lemma of particular interest.

\begin{lemma}
\label{lm:X_law_control}
   Under assumptions of this section, there exists $\delta^\prime > 0$ such that for any $\kappa > 0$ small enough, for any $\eps > 0$ small enough, and for any $x \in \R^d$ such that $|x - a| \leq \kappa/(2 \widetilde{C}_1)$, we have:
    \begin{equation*}
        \sup_{t \leq \exp{\frac{2(H + \delta^\prime)}{\eps}}} \sup_{X \in \widetilde{\X}} \Wass_2 (\Law(X_t), \delta_a) < \kappa / \widetilde{C}_1.
    \end{equation*}
\end{lemma}

This lemma means that for any process $X \in \widetilde{\X}$ its law satisfies $\Wass_2 (\Law(X_t), \delta_a) < \kappa / \widetilde{C}_1$ until time $\exp{\frac{2(H + \delta^\prime)}{\eps}}$. That includes the stopped process $\{Y_{t \wedge S_\kappa} \}_{t\geq 0}$ and thus shows that $S_\kappa > \exp{\frac{2(H + \delta^\prime)}{\eps}}$. That in turn gives us Lemma~\ref{lm:McKean-Vlasov_law_control}, since there is an indistinguishable, up to time $\exp{\frac{2(H + \delta^\prime)}{\eps}}$, copy of $Y$ that belongs to $\widetilde{\X} \subset \X$. Therefore, it is enough to prove Lemma~\ref{lm:X_law_control}. 

To achieve this, we first present the following result, which demonstrates that for any fixed neighborhood around the point $a$, the probability of the process $X_t$ being outside this neighborhood tends to zero as $\eps \to 0$.

\begin{lemma}
\label{lm:X_t_notin_B_rho}
       Under Assumptions~\ref{assu:1}--\ref{assu:4}, there exists $\delta^\prime > 0$ such that for any $\rho > 0$ small enough and for any $\kappa > 0$ small enough we have:
    \begin{equation*}
        \sup_{t \leq \exp{\frac{2(H + \delta^\prime)}{\eps}}} \sup_{x \in B_{\rho/4}(a)} \sup_{X \in \X} \Prob(X_t \notin B_\rho(a)) \xrightarrow[\eps \to 0]{} 0.
    \end{equation*}
\end{lemma}

\begin{proof}
    We will prove the lemma in three steps.

    {\it Step 1.} Let us first define $\delta^\prime > 0$ then, given its value, we will fix some small enough parameters $\rho > 0$ and $\kappa > 0$ and prove the lemma by taking the limit with $\eps \to 0$. As was pointed out on page~\pageref{def:D_enlargement}, there exists an enlargement of the set $\cDc$, that we call $\cDc^e_r \supset \cDc$, that itself satisfies Assumptions~\ref{assu:1}--\ref{assu:4}, but $d(\partial \cDc^e_r, \cDc) > 0$. Since the quasipotential is continuous, its height $\ds H^e_r = \inf_{y \in \partial \cDc^e} Q(y) > H$ and thus we can fix some positive $\delta^\prime$ such that $H + 2\delta^\prime < H^e_r$ for some $r > 0$. That in particular means that by Theorem~\ref{thm:main_result}, for any compact set $K \subset \cDc^e_r$ and for $\tau^e_r(X) := \inf\{t \geq 0: X_t \notin \cDc^e_r\}$, we have
    \begin{equation}
    \label{eq:aux:X_D_e_exit}
        \sup_{x \in K} \sup_{X \in \X} \Prob \left(\tau^e_r(X) < \exp{\frac{2(H + \delta^\prime)}{\eps}} \right) \xrightarrow[\eps \to 0]{} 0. 
    \end{equation}
    
    Given $\delta^\prime$, let us now choose $\rho$. Fix $\rho > 0$ to be small enough such that $\rho < d(\partial \cDc^e, \cDc)$ and $B_{\rho}(a) \subset \cDc$. Let $\rho$ be small enough such that Corollary~\ref{lm:Z_conv} holds for the enlargement $\cDc^e_r$ and with $\delta = \rho/4$. Therefore, there exists a positive time $T^1_\rho$ such that
    \begin{equation}
    \label{eq:aux:Z_conv}
        \sup_{x \in \cDc^e_r} \Prob (Z_{T^1_\rho} \notin B_{\rho/4}(a)) \xrightarrow[\eps \to 0]{} 0.
    \end{equation}

    Let $\rho$ be small enough such that Corollary~\ref{lm:|Z_t-a|>beta_delta} holds with $\delta$ replaced by $\rho/2$, $T$ by $T^1_\rho$ chosen above and $\beta = 3/2$. That gives us:
    \begin{equation}
    \label{eq:aux:B_rho/2_for_Z}
        \sup_{x \in B_{\rho/2}(a)} \Prob\left(\sup_{t \leq T^1_\rho} |Z_t - a| > \frac{3}{4}\rho \right) \xrightarrow[\eps \to 0]{} 0.
    \end{equation}


    Based on the selected $\rho$ and $\delta^\prime$, we will now specify a particular $\kappa$ for which subsequent calculations will be performed, though these could be conducted for any $\kappa$ that is smaller than the fixed one. First, let us use Lemma~\ref{lm:|X_t-Z_t|>delta} with $\delta = \rho/4$ and $T = T_\rho^1$, and chose a small enough $\kappa > 0$ such that 
    \begin{equation}
    \label{eq:aux:|X-Z|>rho/4}
        \sup_{x \in \cDc^e_r} \sup_{X \in \X} \Prob \left(\sup_{t \leq T_\rho^1} |X_t - Z_t| > \frac{\rho}{4}\right) \xrightarrow[\eps \to 0]{} 0.
    \end{equation}

    As the final remark of this step, note that Equations~\eqref{eq:aux:Z_conv}, \eqref{eq:aux:|X-Z|>rho/4} provide the following estimate:
    \begin{equation}
    \label{eq:aux:X_conv}
    \begin{aligned}
        \sup_{x \in \cDc^e_r} \sup_{X \in \X} \Prob (X_{T^1_\rho} \notin B_{\rho/2}(a)) &\leq \sup_{x \in \cDc^e_r} \Prob (Z_{T^1_\rho} \notin B_{\rho/4}(a)) \\
        & \quad + \sup_{x \in \cDc^e_r} \sup_{X \in \X} \Prob \left(\sup_{t \leq T_\rho^1} |X_t - Z_t| > \frac{\rho}{4}\right) \xrightarrow[\eps \to 0]{} 0.        
    \end{aligned}
    \end{equation}
    
    {\it Step 2.} Let us now prove the lemma. First, we show the result for discrete times of the form $k T_\rho^1$, where $k = 1, \dots, \left\lfloor \e^{\frac{2(H + \delta^\prime)}{\eps}} / T_\rho^1 \right\rfloor$. For any $X \in \X$ we have
    \begin{equation*}
    \begin{aligned}
        \sup_{k} \sup_{x \in B_{\rho/4}(a)} \Prob(X_{k T_\rho^1} \notin B_{\rho/2}(a)) &\leq \sup_{k} \sup_{x \in B_{\rho/4}(a)} \Prob(X_{k T_\rho^1} \notin B_{\rho/2}(a), X_{(k - 1)T_\rho^1} \in \cDc^e_r) \\
        & \quad + \sup_{k} \sup_{x \in B_{\rho/4}(a)}  \Prob(X_{(k - 1)T_\rho^1} \notin \cDc^e_r),
    \end{aligned}    
    \end{equation*}
    where the suprema are taken over $k = 1, \dots, \left\lfloor \e^{\frac{2(H + \delta^\prime)}{\eps}} / T_\rho^1 \right\rfloor$. Using Remark~\ref{rem:X_Markov_property}, the expression above is bounded by
    \begin{equation*}
    \begin{aligned}
        \sup_{x \in \cDc^e} \Prob(X_{T_\rho^1} \notin B_{\rho/2}(a)) &+ \sup_{k}\sup_{x \in B_{\rho/4}(a)} \Prob(\tau^e(X) < (k - 1)T_\rho^1) \\
        \leq \sup_{x \in \cDc^e} \Prob(X_{T_\rho^1} \notin B_{\rho/2}(a)) &+ \sup_{x \in B_{\rho/4}(a)} \Prob \left(\tau^e(X) < \exp{\frac{2(H + \delta^\prime)}{\eps}} \right),        
    \end{aligned}
    \end{equation*}
    which tends to zero as $\eps \to 0$ independently of $X \in \X$ due to \eqref{eq:aux:X_conv} and \eqref{eq:aux:X_D_e_exit}. Therefore, we get
    \begin{equation}
    \label{eq:aux:X_k_T_rho_notin_B}
        \sup_{k} \sup_{x \in B_{\rho/4}(a)} \Prob(X_{k T_\rho^1} \notin B_{\rho/2}(a)) \xrightarrow[\eps \to 0]{} 0,
    \end{equation}
    where, again, the supremum is taken over $k = 1, \dots, \left\lfloor \e^{\frac{2(H + \delta^\prime)}{\eps}} / T_\rho^1 \right\rfloor$.

    {\it Step 3.}  To complete the proof, consider $t \in [k T_\rho^1, (k + 1)T_\rho^1]$ for some \\ $k \leq \left\lfloor\exp{\frac{2(H + \delta^\prime)}{\eps}}/T_\rho^1 \right\rfloor$. Then we have
    \begin{equation*}
    \begin{aligned}
        \sup_{X \in \X}\Prob(X_t \notin B_{\rho}(a)) &\leq \sup_{X \in \X} \Prob(X_t \notin B_{\rho}(a), X_{kT_\rho^1} \in B_{\rho/2}(a)) + \sup_{X \in \X} \Prob(X_{kT_\rho^1} \notin B_{\rho/2}(a)) \\
        & \leq \sup_{x \in B_{\rho/2}(a)} \sup_{X^\prime \in \X} \Prob(X^\prime_{t - kT^1_\rho} \notin B_{\rho}(a)) + \sup_{X \in \X} \Prob(X_{kT_\rho^1} \notin B_{\rho/2}(a)) \\
        & \leq \sup_{x \in B_{\rho/2}(a)} \sup_{X^\prime \in \X} \Prob(\sup_{t \leq T_\rho^1}|X^\prime_t - a| > \rho) + \sup_{X \in \X} \Prob(X_{kT_\rho^1} \notin B_{\rho/2}(a)).
    \end{aligned} 
    \end{equation*}

    The expression above is less or equal than
    \begin{equation*}
    \begin{aligned}
        \sup_{x \in B_{\rho/2}(a)} \sup_{X^\prime \in \X} \Prob\left(\sup_{t \leq T_\rho^1}|X^\prime_t - Z_t| > \frac{\rho}{4}\right) + \sup_{x \in B_{\rho/2}(a)} \Prob \left(\sup_{t \leq T_\rho^1}|Z_t - a| > \frac{3\rho}{4} \right) \\
        + \sup_{X \in \X} \Prob(X_{kT_\rho^1} \notin B_{\rho/2}(a)),        
    \end{aligned}
    \end{equation*}
    which tends to zero with $\eps \to 0$ by Equations~\eqref{eq:aux:|X-Z|>rho/4}, \eqref{eq:aux:B_rho/2_for_Z} and \eqref{eq:aux:X_k_T_rho_notin_B}. Moreover,  convergence in \eqref{eq:aux:X_k_T_rho_notin_B} does not depend on $k$. This proves the lemma.
\end{proof}

We are now prepared to prove Lemma~\ref{lm:X_law_control}, which establishes control of the law $\mu_t$ and, as a consequence, provides a proof of Theorem~\ref{thm:McKean_exit_time}.

\begin{proof}[Proof of Lemma~\ref{lm:X_law_control}]
    We prove this lemma in several steps.
    
    \textit{Step 1}. Let us first fix the constants that we use in the proof. In the following steps, we will only adjust $\eps$. Take some positive $\overline{\rho}$ that satisfies $\overline{\rho} < R/2 \wedge 1$, where $R$ is defined in Assumption~\ref{assu:MV:control2}. Decrease it if necessary so that $\widetilde{K}(2\overline{\rho}) > \widetilde{K}_2  := (\widetilde{K}_1 + \widetilde{K})/2$ (see Fig.~\ref{fig:K_positions}), which is possible since $\widetilde{K}(r) \xrightarrow[r \to 0]{} \widetilde{K}$. Obviously, for any $\rho < \overline{\rho}$ we would still have $\widetilde{K}(2 \rho) > \widetilde{K}_2$. This seemingly arbitrary choice of $\overline{\rho}$ will be apparent soon.

    \begin{figure}[b]
        \centering
        \includegraphics{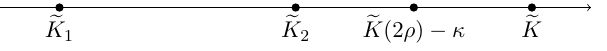}
        \caption{Depiction of the relationship between $\widetilde{K}_1$, $\widetilde{K}_2$, $(\widetilde{K}(2\rho) - \kappa)$, and $\widetilde{K}$.}
        \label{fig:K_positions}
    \end{figure}

    Fix $\delta^\prime$ as in Lemma~\ref{lm:X_t_notin_B_rho}. Take some $\rho$ and $\kappa$ that are small enough so that Lemma~\ref{lm:X_t_notin_B_rho} holds. decrease $\rho$ if necessary to have $\rho < \overline{\rho}$. Decrease $\kappa$ if necessary, to satisfy Lemma~\ref{lm:X_t_notin_B_rho} with the new value of $\rho$. This update is of importance since in this Lemma the upper bound for the values of $\kappa$ for which the result of the lemma holds depends on the choice of $\rho$. The parameter $\rho$ is finally fixed and will not be adjusted anymore, though we will still decrease $\kappa$. This will not change the fact that, due to Lemma~\ref{lm:X_t_notin_B_rho}, we have:
    \begin{equation}
    \label{eq:aux:lm:X_t_notin_B_rho}
        \sup_{t \leq \exp{\frac{2(H + \delta^\prime)}{\eps}}} \sup_{x \in B_{\rho/4}(a)} \sup_{X \in \X} \Prob(X_t \notin B_{\rho}(a)) \xrightarrow[\eps \to 0]{} 0.
    \end{equation}

    Decrease $\kappa$ to have $\widetilde{K}(2 \rho) - \kappa > \widetilde{K}_2$. Now we can see the reason behind choosing $\rho$ this way. Since $\rho < R/2$, for any $x \in B_{2 \rho}(a)$ and for any $t \geq 0$ we have
    \begin{equation}
    \label{eq:aux:nabla_Drift_bound}
        \nabla \Drift(t, x) = \nabla \drift(x) + (\nabla \Drift(t, x) - \nabla b(x)) \preceq (-\widetilde{K}(2 \rho) + \kappa)\Id \leq -\widetilde{K}_2 \Id. 
    \end{equation}

    Finally, note that $B_{2 \rho}(a)$ being a positively invariant for the flow $\dot{\phi_t} = \drift(\phi_t)$ domain, satisfies the conditions of Assumptions~\ref{assu:1}--\ref{assu:4}, thus the exit-time results as in Theorem~\ref{thm:main_result} hold inside it. In particular, if we fix 
    \begin{equation}
    \label{eq:def:T_1}
        T_1 := \frac{2 \widetilde{C}_1^2}{(\widetilde{K}_2 - \widetilde{K}_1)\kappa},
    \end{equation} 
    a positive time the meaning of which will be apparent in Step 2, then, since it does not depend on $\eps$, by Theorem~\ref{thm:main_result}, we have:
    \begin{equation}
    \label{eq:aux:|X_t - a|_sqrt_rho}
        \sup_{x \in B_{\rho}(a)} \sup_{X \in \widetilde{\X}} \Prob (\sup_{t \in [0 , T_1]} |X_t - a| > 2\rho) \xrightarrow[\eps \to 0]{} 0. 
    \end{equation}
    With this remark, we conclude specifying the constants used in the proof and are ready to proceed to the next step.

    \textit{Step 2}. Define the following stopping time $\theta := \inf\{t \geq 0: X_t \notin B_{2 \rho}(a)\}$. In this step, we prove, first, that there exists $1/2 < m < 1$ such that for any $\eps$ small enough we have
    \begin{equation}
    \label{eq:aux:control_of_xi}
            |x - a| \leq \rho \Rightarrow \sup_{X \in \X} \Wass_2(\Law(X_{T_1 \wedge \theta}), \delta_a) < m \left(\frac{\kappa}{\widetilde{C}_1} \right),
    \end{equation}
    and, second, we have:
    \begin{equation}
    \label{eq:aux:control_of_xi_2}
        |x - a| \leq \frac{\kappa}{2\widetilde{C}_1} \Rightarrow \forall t \in [0, T_1]: \sup_{X \in \X} \Wass_2(\Law(X_{t \wedge \theta}), \delta_a) \leq \left(\frac{\kappa}{2\widetilde{C}_1} \right).
    \end{equation}

     In order to do so, define for any $X \in \widetilde{\X}(\eps, \kappa, x)$ the following:
    \begin{equation*}
        \xi(t) := \E |X_{t \wedge \theta} - a|^2 = \Wass_2^2(\Law(X_{t \wedge \theta}), \delta_a).
    \end{equation*}
    
    The Itô lemma gives us that for any $t \leq T_1$: 
    \begin{equation*}
    \begin{aligned}
        \frac{1}{2} \xi(t) = \E \int_{0}^{t \wedge \theta} \!\!\left[ \langle X_s - a, \sqrt{\eps}\Diff(s, X_s)\dd{W_s} + \Drift(s, X_s)\dd{s} \rangle + \frac{\eps}{2} \tr(\Diff^T(s, X_s) \Diff(s, X_s)) \dd{s} \right]. 
    \end{aligned}
    \end{equation*}

    Note that the expectation of the martingale part is equal to zero. Note also that, by the definition of $\widetilde{\X}$, $\|\Diff(s, x) - \diff(x)\| \leq \kappa < 1$ and, by Assumption~\ref{assu:4}, $\|\sigma(x)\| \leq \lambda$. Therefore, there exists a universal constant $C > 0$ such that:
    \begin{equation*}
    \begin{aligned}
        \frac{1}{2} \xi(t) &\leq \E \left[ \int_{0}^{t \wedge \theta} \langle X_s - a, \Drift(s, X_s)\dd{s} \rangle \right] + C\eps t \\ 
        & = \E \left[ \int_{0}^{t \wedge \theta} \langle X_s - a, \Drift(s, X_s) + \Drift(s, a)\rangle \dd{s} \right] + \E \left[ \int_{0}^{t \wedge \theta} \langle X_s - a, \Drift(s, a) \rangle \dd{s} \right] + C\eps t.
    \end{aligned}
    \end{equation*}

    By Equation~\eqref{eq:aux:nabla_Drift_bound} and using the Cauchy–Schwarz inequality, we get:
    \begin{equation*}
    \begin{aligned}
        \frac{1}{2} \xi(t) &\leq - \widetilde{K}_3(\rho) \E \left[ \int_{0}^{t \wedge \theta}| X_s - a|^2 \dd{s} \right] + \E\left[\int_0^{t \wedge \theta}|X_s - a|  \left|\Drift(s, a) \right| \dd{s} \right] + C\eps t.
    \end{aligned}
    \end{equation*}

    Note that for any $t, u \in \R_+$ and any function $f$, the following holds $\int_0^{t \wedge u} f_s \dd{s} = \int_0^t f_{s \wedge u} \dd{s} - f_u \1\{u < t\} (t - u)$. Moreover, the definition of $\widetilde{\X}$ gives us control of the term $|\Drift(s, a)|$. We get:
    \begin{equation*}
    \begin{aligned}
        \frac{1}{2} \xi(t) &\leq - \widetilde{K}_2 \E \left[ \int_{0}^{t}| X_{s \wedge \theta} - a|^2 \dd{s} \right] + \widetilde{K}_2 \E \left[ | X_{\theta} - a|^2 \1\{\theta < t\} \right] t \\
        & \quad + \widetilde{K}_1 \frac{\kappa}{\widetilde{C}_1} \E\left[\int_0^{t \wedge \theta}|X_s - a| \dd{s} \right] + C\eps t.
    \end{aligned}
    \end{equation*}
    In the following, we use the definition of $\theta$, which gives $|X_{\theta} - a|^2 = 4 \rho^2$ a.s., and the fact that $\int_0^{t \wedge u} f_s \dd{s} \leq \int_0^t f_{s \wedge u} \dd{s}$ for $f \geq 0$ to get:
    \begin{equation*}
    \begin{aligned}
        \frac{1}{2} \xi(t) &\leq - \widetilde{K}_2 \int_{0}^{t} \E| X_{s \wedge \theta} - a|^2 \! \dd{s}  + \widetilde{K}_1 \frac{\kappa}{\widetilde{C}_1} \int_0^{t} \E|X_{s\wedge \theta} - a|\! \dd{s} + \Big(C\eps + 4\widetilde{K}_2 \rho^2 \;\Prob(\theta < t) \Big) t\\
        & = - \widetilde{K}_2 \int_{0}^{t} \xi(s) \! \dd{s}  + \widetilde{K}_1 \frac{\kappa}{\widetilde{C}_1} \int_0^{t} \!\sqrt{\xi(s)}\! \dd{s} + \Big(C\eps + 4\widetilde{K}_2 \rho^2 \;\Prob( \theta < t) \Big) t.
    \end{aligned}
    \end{equation*}
    Recall that we consider only $t \leq T_1$ and thus, by Equation~\eqref{eq:aux:|X_t - a|_sqrt_rho}, $\Prob(\theta < t) \leq \Prob(\theta < T_1) \xrightarrow[\eps \to 0]{} 0$. Therefore, there exists $\alpha(\eps)$ that tends to 0 when $\eps \to 0$ such that
    \begin{equation*}
         \xi(t) \leq - 2\widetilde{K}_2 \int_{0}^{t} \xi(s) \! \dd{s}  + 2\widetilde{K}_1 \frac{\kappa}{\widetilde{C}_1} \int_0^{t} \!\sqrt{\xi(s)}\! \dd{s} + \alpha(\eps)t.
    \end{equation*}
    Note that $X$, being a strong solution of the form \eqref{eq:def_main_process}, gives $\xi$ that is continuous. That means that if we take $\psi$ the solution of
    \begin{equation*}
    \begin{cases}
        \dot{\psi}(t) = - 2\widetilde{K}_2 \psi(t)  + 2\widetilde{K}_1 \frac{\kappa}{\widetilde{C}_1} \sqrt{\psi(t)} + \alpha(\eps), \\
        \psi(0) = \rho;
    \end{cases}
    \end{equation*}
    then $\xi(t) \leq \psi(t)$ for any $t \geq 0$. Let us study the graph of $\psi$. Simple computations show that $\dot{\psi}(t) < 0$ whenever we have
    \begin{equation*}
    \begin{aligned}
        \sqrt{\psi(t)} &> \frac{\widetilde{K}_1 \frac{\kappa}{\widetilde{C}_1} + \sqrt{(\widetilde{K}_1 \frac{\kappa}{\widetilde{C}_1})^2 + \alpha(\eps)}}{2 \widetilde{K}_2} = \frac{1}{2}\left(1 + \sqrt{1 + \widetilde{C}_1^2\alpha(\eps)/(\widetilde{K}_1\kappa)^2} \right) \frac{\widetilde{K}_1}{\widetilde{K}_2} \frac{\kappa}{\widetilde{C}_1} \\
        &=: A(\eps) \frac{\widetilde{K}_1}{\widetilde{K}_2} \frac{\kappa}{\widetilde{C}_1}.        
    \end{aligned}
    \end{equation*}
    See Fig. \ref{fig:psi_dot_psi_dependency} for a graphical depiction. By the definition of $\widetilde{K}_2$, we have $\widetilde{K}_1/\widetilde{K}_2 < 1$. Let us define $m := (\widetilde{K}_1/\widetilde{K}_2 + 1)/2$ and let us choose $\eps$ to be small enough such that $A(\eps)\widetilde{K}_1/\widetilde{K_2} < m < 1$ (see Fig.~\ref{fig:psi_t_conv}).

   \begin{figure}[t]
        \centering
        \includegraphics{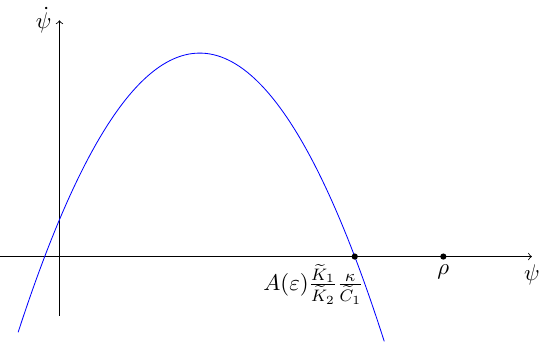}
        \caption{The dependency between $\psi(t)$ and $\dot{\psi}(t)$.}
        \label{fig:psi_dot_psi_dependency}
    \end{figure}

    \begin{figure}[t]
        \centering
        \includegraphics{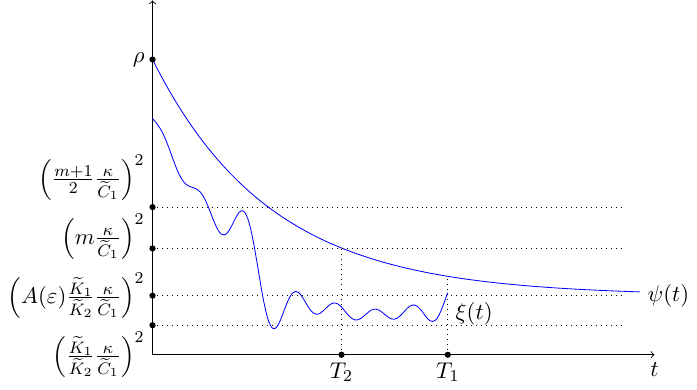}
        \caption{Convergence of $\psi$.}
        \label{fig:psi_t_conv}
    \end{figure}

    Let us now estimate the time of convergence $T_2 := \inf\{t \geq 0: \psi(t) \leq (m \kappa/ \widetilde{C}_1)^2\}$. By the definition of $\psi$, we have
    \begin{equation*}
        \psi(T_2) - \rho = - 2\widetilde{K}_2 \int_0^{T_2} \psi(s) \dd{s} + 2 \widetilde{K}_1 \frac{\kappa}{\widetilde{C}_1}\int_0^{T_2} \sqrt{\psi(s)}\dd{s} + \alpha(\eps).
    \end{equation*}
    Note that for any $t \leq T_2$, the derivative $\dot{\psi}(t) < 0$, which means that the smallest value of $\psi$ on the time interval $[0, T_2]$ is actually obtained at $T_2$. Since both the left- and right-hand sides of the equation above are negative, that gives us the following estimate:
    \begin{equation*}
    \begin{aligned}
        \psi(T_2) - \rho &\leq \left( - 2\widetilde{K}_2 \psi(T_2) + 2 \widetilde{K}_1 \frac{\kappa}{\widetilde{C}_1} \sqrt{\psi(T_2)} \right) T_2 + \alpha(\eps) \\
        & = 2  \left( - \widetilde{K}_2 m + \widetilde{K}_1 \right) m \left(\frac{\kappa}{\widetilde{C}_1}\right)^2 T_2 + \alpha(\eps).
    \end{aligned}
    \end{equation*}
    We use the definition of $m$, i.e. $m = \frac{\widetilde{K}_1}{2\widetilde{K}_2} + \frac{1}{2}$ and a simple bound $m > \frac{1}{2}$ and get
    \begin{equation*}
        \psi(T_2) - \rho \leq \left( - \frac{\widetilde{K}_1}{2} - \frac{\widetilde{K}_2}{2} + \widetilde{K}_1 \right) \left(\frac{\kappa}{\widetilde{C}_1} \right)^2 T_2 + \alpha(\eps) \leq \frac{1}{2} (\widetilde{K}_1 - \widetilde{K}_2) \left(\frac{\kappa}{\widetilde{C}_1}\right)^2 T_2 + \alpha(\eps).
    \end{equation*}
    Since $\widetilde{K}_1 < \widetilde{K}_2$, we get:
    \begin{equation*}
        T_2 \leq \frac{\rho + \alpha(\eps) - \psi(T_2)}{\frac{1}{2}(\widetilde{K}_2 - \widetilde{K}_1) \left( \frac{\kappa}{\widetilde{C}_1}\right)^2} < \frac{2 \widetilde{C}_1^2}{(\widetilde{K}_2 - \widetilde{K}_1) \kappa},
    \end{equation*}
    where we obtain the last inequality by choosing $\eps$ to be small enough such that $\alpha(\eps) < \psi(T_2) = (m \kappa/ \widetilde{C}_1)^2$. Recall our choice of $T_1$ in \eqref{eq:def:T_1}. It was chosen to be exactly equal to the right-hand side of the inequality above. That means that
    \begin{equation*}
        T_2 < T_1,
    \end{equation*}
    which in turns finally gives us that at the point of time $T_1$, we have:
    \begin{equation*}
        \xi(T_1) \leq \psi(T_1) \leq m^2 \left(\frac{\kappa}{\widetilde{C}_1}\right)^2 < \left(\frac{\kappa}{\widetilde{C}_1}\right)^2.
    \end{equation*}
    Moreover, note that, if $\xi(0) < \left(\frac{\kappa}{\widetilde{C}_1}\right)^2$, then we can easily show that the function $\psi$ starting at $\psi(0) = \left(\frac{\kappa}{\widetilde{C}_1}\right)^2$ will decrease with time, which gives that
    \begin{equation*}
        \xi(t) < \left(\frac{\kappa}{\widetilde{C}_1}\right)^2
    \end{equation*}
    for any $t \leq T_1$. All the estimates in this step are made independently of the choice of $X \in \widetilde{\X}$, which finally gives us \eqref{eq:aux:control_of_xi} and \eqref{eq:aux:control_of_xi_2}.
    
    \textit{Step 3}. In this step, we will remove $\theta$ from the result of the previous step (Equations~\eqref{eq:aux:control_of_xi}, \eqref{eq:aux:control_of_xi_2}). Namely, we will prove three following equations.
    \begin{equation}
    \label{eq:aux:control_of_xi_no_theta}
            |x - a| \leq \rho \Rightarrow \sup_{X \in \widetilde{\X}} \Wass_2(\Law(X_{T_1}), \delta_a) < \left(\frac{m + 1}{2} \right)\frac{\kappa}{\widetilde{C}_1}, \text{ and}     
    \end{equation}
    
    \begin{equation}
    \label{eq:aux:control_of_xi_no_theta_2}
            |x - a| \leq \frac{\kappa}{2\widetilde{C}_1} \Rightarrow \forall t \in [0, T_1]: \sup_{X \in \widetilde{\X}} \Wass_2(\Law(X_{t}), \delta_a) < \frac{\kappa}{\widetilde{C}_1}.  
    \end{equation}

    The trade-off of removing $\theta$ is that now the upper bound in both equations increased. In order to prove this, for any $x \in \cDc$ and $X \in \widetilde{\X}$, consider
    \begin{equation*}
    \begin{aligned}
        \E|X_{t} - a|^2 &= \E \left[|X_{t} - a|^2 \big(\1_{\{t \leq \theta\}} + \1_{\{t > \theta\}}\big)\right] \\
        &\leq \E \left[|X_{t \wedge \theta} - a|^2\right] + \sqrt{\E |X_{t} - a|^4 } \sqrt{\Prob(\theta < t)}.
    \end{aligned}
    \end{equation*}

    The probability $\Prob(\theta < t) \leq \Prob(\theta < T_2)$ tends to zero as $\eps \to 0$ by Equation~\eqref{eq:aux:|X_t - a|_sqrt_rho}. Since $X \in \widetilde{\X}$, the expectation $\E|X_{t} - a|^4 \leq  M$ (see Definition~\ref{def:X_tilde}). Lastly, $\E \left[|X_{t \wedge \theta} - a|^2\right]$ is equal to $\xi(t)$ by the definition of the latter. Therefore, we get
    \begin{equation}
        \E|X_{t} - a|^2 \leq \xi(t) + \sqrt{M}\alpha(\eps),
    \end{equation}
    for some $\alpha$ that tends to zero when $\eps \to 0$. That proves Equation~\eqref{eq:aux:control_of_xi_no_theta} by choosing $\eps$ small enough. In order to prove \eqref{eq:aux:control_of_xi_no_theta_2}, we use \eqref{eq:aux:control_of_xi_2} and the same reasoning as above.

    \textit{Step 4}. This is the final step of the proof. Recall that we have to show that 
    \begin{equation*}
        \sup_{x \in B_{\frac{\kappa}{2\widetilde{C}_1}}(a)} \sup_{X \in \widetilde{\X}} \Wass_2(\Law(X_t), \delta_a) < \frac{\kappa}{\widetilde{C}_1},
    \end{equation*}
    for all $t \in \left[0, \exp{\frac{2(H + \delta^\prime)}{\eps}} \right]$. We will treat two different cases separately; when $t \leq T_1$ and when $t \in \left(T_1, \exp{\frac{2(H + \delta^\prime)}{\eps}} \right]$.
    
    \textit{4.1}. Let $t \leq T_1$. Note that, since the starting point $|x - a| \leq \kappa/(2 \widetilde{C}_1)$, by Equation~\eqref{eq:aux:control_of_xi_no_theta_2}, we simply have
    \begin{equation*}
        \forall s \in [0, T_1]: \sup_{X \in \widetilde{\X}} \Wass_2(\Law(X_t), \delta_a) < \frac{\kappa}{C}.
    \end{equation*}

    \textit{4.2}. Now take some $t \in \left(T_1, \exp{\frac{2(H + \delta^\prime)}{\eps}} \right]$. For any $x \in B_{\frac{\kappa}{2 \widetilde{C}_1}}(a)$ and for any $X \in \widetilde{\X}(\eps, \kappa, x)$, we have
    \begin{equation*}
    \begin{aligned}
        &\E|X_{t} - a|^2 = \E \left[|X_{t} - a|^2 \left(\1_{\{X_{t - T_1} \in B_{\rho}(a)\}} + \1_{\{X_{t - T_1} \notin B_{\rho}(a)\}} \right) \right] \\
        & \leq \E \! \left[\1_{\{X_{t - T_1} \in B_{\rho}(a)\}} \E\big[|X_{t} - a|^2 |\mathcal{F}_{t - T_1}\big] \right] + \sqrt{M} \sqrt{\Prob\left(X_{t - T_1} \notin B_{\rho}(a)\right)},
    \end{aligned}
    \end{equation*}
    where we used the Cauchy–Schwarz inequality and, once again, the definition of $\widetilde{\X}$ to have $\E[|X_{t} - a|^4] \leq M$. By Lemma~\ref{lm:X_t_notin_B_rho}, there exists $\alpha(\eps)$ that tends to zero as $\eps \to 0$ such that
    \begin{equation*}
        \sqrt{\Prob\left(X_{t - T_1} \notin B_{\rho}(a)\right)} \leq \alpha(\eps).
    \end{equation*}
    Moreover, $\alpha$ is independent of the choice of $t$, $x$ or $X$ that we made (see Lemma~\ref{lm:X_t_notin_B_rho}). Therefore, we get the following:
    \begin{equation*}
    \begin{aligned}
        \E|X_{t} - a|^2 &\leq \E \left[\1_{\{X_{t - T_1} \in B_{\rho}(a)\}} \times \sup_{X^\prime \in \widetilde{\X}(\eps, \kappa, X_{t - T_1})} \E |X^\prime_{T_1} - a|^2 \right] + \sqrt{M} \alpha(\eps)\\        
        &\leq \sup_{x \in B_{\rho}(a)} \sup_{X^\prime \in \widetilde{\X}(\eps, \kappa, x)} \E |X^\prime_{T_1} - a|^2 + \sqrt{M} \alpha(\eps). 
    \end{aligned}
    \end{equation*}

    By Equation~\eqref{eq:aux:control_of_xi_no_theta}, the first item is bounded by $\left(\frac{m + 1}{2}\right)\frac{\kappa}{\widetilde{C}_1}$ and for the second we decrease $\eps$ to finally get the following uniform upper bound:
    \begin{equation}
    \label{eq:aux:|X_{kT_1} - a|^2_bound}
        \sup_{x \in B_{\frac{\kappa}{2 \widetilde{C}_1}}(a)} \sup_{X \in \widetilde{\X}(\eps, \kappa, x)} \E|X_{t} - a|^2 < \frac{\kappa}{\widetilde{C}_1},
    \end{equation}
    which proves the lemma.
    
\end{proof}

This proof concludes the main part of the paper. In the following, we present some auxiliary results related to the process $Z$.

\appendix

\section{Appendix}
\label{s:Appendix}

In this section, we will provide some facts about the linear diffusion $Z$ defined by Equation~\eqref{eq:def_process_Z}. The first lemma describes the classical large deviation principle for the process $Z$. Let us recall the notion of the large deviation principle (LDP). See \cite{DZ} for more details. Consider the following definition.

\begin{definition}
\label{def:LDP}
    Let $\mathcal{X}$ denote a Banach space and $\mathcal{B}$ denote the Borel $\sigma$-algebra defined on it. Let $\{\nu_\eps\}_{\eps > 0}$ be a family of measures on the measurable set $(\mathcal{X}, \mathcal{B})$. We say that the family $\{\nu_\eps\}_{\eps > 0}$ \textit{satisfies the LDP with a good rate function $I$} if
    \begin{enumerate}
    \item There exists a lower semicontinuous function $I: \mathcal{X} \to [0, \infty]$ such that the level sets $\{ x \in \mathcal{X}: I(x) \leq \alpha\}$ are compact subsets of $\mathcal{X}$,
        \item For any $\Gamma \in \mathcal{B}$, the family of measures $\{\nu_\eps\}_{\eps > 0}$ satisfies
        \begin{equation*}
            - \inf_{x \in \Int (\Gamma)} I(x) \leq \liminf_{\eps \to 0}\frac{\eps}{2} \log \nu_\eps(\Gamma) \leq \limsup_{\eps \to 0} \frac{\eps}{2} \log \nu_\eps(\Gamma) \leq - \inf_{x \in \Cl (\Gamma)} I(x),
        \end{equation*}
        where $\Int$ denotes interior and $\Cl$ the closure of sets.
    \end{enumerate}
\end{definition}

In the case of diffusion \eqref{eq:def_process_Z}, the classical result that we provide here gives us the behavior of trajectories of $Z$ as elements of $C[0, T]$. For any $f \in C[0, T]$, consider the following function:
\begin{equation}
\label{eq:rate_function_Z}
    I_T^x(f) := \int_0^T \left\langle \dot{f}_s - \drift(f_s), (\diff\diff^*)^{-1} (f_s) \big(\dot{f}_s - \drift(f_s) \big) \right\rangle \dd{s},
\end{equation}
where $x \in \R^d$ is the starting point of $Z$ provided by \eqref{eq:def_main_process}. Note that $I^x$ is a function that satisfies condition $1$ of the Definition~\ref{def:LDP} above (see Section~5.6 of \cite{DZ}).

\begin{lemma}[{\cite[Theorem 5.6.7]{DZ}}]
\label{lm:Z_LDP}
    Under Assumptions~\ref{assu:1}--\ref{assu:4}, for any $T > 0$, the family of measures $\{\nu_\eps\}_{\eps > 0}$ induced by $\{Z_s\}_{s \in [0, T]}$ on $C[0, T]$ satisfies the LDP with the good rate function \eqref{eq:rate_function_Z}.
\end{lemma}

More precisely, by the family of measures, we mean that if we take the process $Z$ solution of \eqref{eq:def_main_process} as a measurable function $Z: \Omega \to C[0, T]$ that for each $\omega \in \Omega$ returns the trajectory of $Z$ that corresponds to it, then $\nu_\eps = Z_{\#} \Prob$ for each respective $\eps > 0$. In particular, that gives us:
\begin{enumerate}
    \item for any closed $F \subset C[0, T]$:
    \begin{equation*}
        \limsup_{\eps \to 0} \frac{\eps}{2} \log \Prob (Z \in F) \leq - \inf_{f \in F} I^x_T(f)
    \end{equation*}
    \item for any open $G \subset C[0, T]$:
    \begin{equation*}
        \liminf_{\eps \to 0} \frac{\eps}{2} \log \Prob (Z \in G) \geq - \inf_{f \in G} I^x_T(f).
    \end{equation*}
\end{enumerate}

Observe that the rate function $I^x \geq 0$. Moreover, zero is obtained only for the unique solution of
\begin{equation}
\label{eq:def:phi_t}
    \phi^x_t = x + \int_0^t \drift(\phi^x_s)\dd{s}.
\end{equation}
For any other $f \neq \phi^x$, we have $I^x_T(f) > 0$ and thus for any closed $F \subset C[0, T]$ with $\phi^x \notin F$, we have $\ds \inf_{f \in F} I^x_T(f) > 0$. This observtion is important for the following inequality that is used throughout the paper. For any $\delta > 0$, we have
\begin{equation}
\label{eq:aux:Z-phi_geq_delta}
    \limsup_{\eps \to 0} \frac{\eps}{2} \log \Prob(\|Z - \phi^x\|_\infty \geq \delta) \leq - \inf_{f: \|f - \phi^x\|_{\infty} \geq \delta} I^x_T(f) < 0.
\end{equation}

Yet, in the case of the diffusion process $Z$, more interesting estimates could be provided. Consider the following.

\begin{lemma}[{\cite[Corollary 5.6.15]{DZ}}]
\label{lm:Z_LDP_compact_starting_set}
For any compact $K \subset \R^d$ we have
\begin{enumerate}
    \item for any closed $F \subset C[0, T]$:
    \begin{equation*}
        \limsup_{\eps \to 0} \frac{\eps}{2} \log \sup_{x \in  K} \Prob (Z \in F) \leq - \inf_{x \in K}\inf_{f \in F} I^x_T(f)
    \end{equation*}
    \item for any open $G \subset C[0, T]$:
    \begin{equation*}
        \liminf_{\eps \to 0} \frac{\eps}{2} \log \inf_{x \in K} \Prob (Z \in G) \geq - \sup_{x \in K} \inf_{f \in G} I^x_T(f).
    \end{equation*}
\end{enumerate}

\end{lemma}

Obtaining an equation in the form of \eqref{eq:aux:Z-phi_geq_delta} uniformly in $x$ is challenging. However, if an exponential rate of convergence to zero is not required, achieving this result becomes more straightforward. Consider the following lemma:

\begin{lemma}
\label{lm:Z_phi_control}
    For any $\delta > 0$ and for any positive time $T$, we have
    \begin{equation*}
        \sup_{x \in \R^d}\Prob(\sup_{ t \leq T} |Z_t - \phi_t^x| \geq \delta) \xrightarrow[\eps \to 0]{} 0.
    \end{equation*}
\end{lemma}

\begin{proof}
    We use the Lipschitz continuity of $b$ and simply get:
    \begin{equation*}
        |Z_t - \phi_t^x| \leq \sqrt{\eps} \left| \int_0^t \diff(Z_s)\dd{W_s}\right| + C \int_0^t |Z_s - \phi_s^x|\dd{s},  
    \end{equation*}
    where $C > 0$ is the Lipschitz constant and thus independent of the starting point $x$. Taking supremum, we get:
    \begin{equation*}
        \sup_{u \leq t} |Z_u - \phi_u^x| \leq \sqrt{\eps} \sup_{u \leq t} \left| \int_0^u \diff(Z_s)\dd{W_s}\right| + C \int_0^t \sup_{u \leq s} |Z_u - \phi_u^x| \dd{s}. 
    \end{equation*}
    By Gronwall's inequality, we get
    \begin{equation*}
        \sup_{u \leq T} |Z_u - \phi_u^x| \leq \sqrt{\eps}\, \e^{CT} \sup_{u \leq T} \left| \int_0^u \diff(Z_s)\dd{W_s}\right|.
    \end{equation*}
    All the inequalities above hold a.s. Therefore, we get
    \begin{equation*}
    \begin{aligned}
        \Prob(\sup_{t \leq T}|Z_t - \phi_t^x| \geq \delta) &\leq \Prob\left( \sqrt{\eps}\sup_{t \leq T} \left| \int_0^t \diff(Z_s)\dd{W_s}\right| \geq \delta \e^{-CT} \right) \\
        & \leq \sqrt{\eps}\, \E\left[\sup_{t \leq T} \left| \int_0^t \diff(Z_s)\dd{W_s}\right|\right] \e^{CT}/\delta.
    \end{aligned}
    \end{equation*}
     By the Burkholder–Davis–Gundy inequality, we get:
     \begin{equation*}
         \Prob\left(\sup_{t \leq T}|Z_t - \phi_t^x| \geq \delta \right) \leq \sqrt{\eps}\, \E\left[ \sqrt{\int_0^T \diff^2(Z_s)\dd{s}}\right] \e^{CT}/\delta,
     \end{equation*}
     which due to Assumption~\ref{assu:4} (i.e. $\|\sigma(x)\| \leq \lambda$) gives us
     \begin{equation*}
         \Prob\left(\sup_{t \leq T}|Z_t - \phi_t^x| \geq \delta \right) \leq \sqrt{\eps} \lambda \sqrt{T} \e^{CT}/\delta \xrightarrow[\eps \to 0]{} 0.
     \end{equation*}
     The upper bound for the above probability does not depend on the starting point $x \in \R^d$ thus proving the lemma.
\end{proof}

Since $\cDc$ is assumed to be a positively invariant domain and $a$ is the unique attractor within it, the result above implies, in particular, that the process $Z$ will also converge to $a$ within a deterministic time. Consider the following corollary.

\begin{corollary}
\label{lm:Z_conv}
    Under Assumptions~\ref{assu:1}--\ref{assu:4}, for any $\delta > 0$ small enough there exists a positive time $T_\delta$ such that
    \begin{equation*}
        \sup_{x \in \cDc} \Prob \left(Z_{T_\delta} \notin B_{\delta}(a) \right) \xrightarrow[\eps \to 0]{} 0.
    \end{equation*}
\end{corollary}

\begin{proof}
    Fix $\delta > 0$ such that $B_\delta(a) \subset \cDc$. Since $\overline{\cDc}$ is a compact that is positively invariant for the dynamical system \eqref{eq:def:phi_t}, there exists $T_\delta$ such that $\ds \sup_{x \in \cDc} |\phi_{T_\delta}^x - a| \leq \delta/2$.

    We use Lemma~\ref{lm:Z_phi_control} with $\delta/2$ instead of $\delta$ and $T = T_\delta$. We finally get
    \begin{equation*}
        \sup_{x \in \cDc} \Prob \left(Z_{T_\delta} \notin B_{\delta}(a) \right) \leq \sup_{x \in \cDc} \Prob \left(|Z_{T_\delta} - \phi_{T_\delta}^x| > \delta/2 \right) \leq \sup_{x \in \cDc} \Prob \left( \sup_{t \leq T_\delta}|Z_{t} - \phi_{t}^x| \geq \delta/2 \right),
    \end{equation*}
    which tends to zero as $\eps \to 0$.
\end{proof}

In addition to Corollary~\ref{lm:Z_conv}, we can also state that if the process $Z$ starts sufficiently close to the attractor $a$, it is expected not to deviate far from it over any given deterministic time. We provide a rigorous formulation of this statement in the following corollary.

\begin{corollary}
\label{lm:|Z_t-a|>beta_delta}
    Under Assumptions~\ref{assu:1}--\ref{assu:4}, for any $\delta > 0$ small enough, for any positive time $T$ and for any $\beta > 1$ we have 
    \begin{equation*}
        \sup_{x \in B_{\delta}(a)} \Prob \left(\sup_{t \leq T}|Z_t - a| > \beta \delta \right) \xrightarrow[\eps \to 0]{} 0.
    \end{equation*}
\end{corollary}

\begin{proof}
    Let us fix $\delta > 0$ to be small enough so that $\langle \drift(x), x - a \rangle < 0$, meaning that the vector field $\drift$ points inside $B_{\delta}(a)$. This is possible due to Assumption~\ref{assu:3}. Then $B_{\delta}(a)$ is positively for the flow $\dot{\phi}^x_t = \drift(\varphi_t)$ and for any $x \in B_{\delta}(a)$ we have $\{\phi_t^x\}_{t \geq 0} \subset B_{\delta}(a)$. 
    
    Choose any positive $T$ and $\beta > 1$. Since $\phi$ always stays inside $B_\delta(a)$ and by Lemma~\ref{lm:Z_phi_control}, we finally get
    \begin{equation*}
        \sup_{x \in B_{\delta}(a)} \Prob \left(\sup_{t \leq T}|Z_t - a| > \beta \delta \right) \leq \sup_{x \in B_{\delta}(a)} \Prob \left(\sup_{t \leq T} |Z_t - \phi_t^x| > (\beta - 1)\delta \right) \xrightarrow[\eps \to 0]{} 0.
    \end{equation*}
\end{proof}

In the following, we present some known in the literature results about the process $Z$ derived using large deviations techniques. For further details on these methods, we refer to {\cite[Chapter 5]{DZ}}. 

Based on the estimates already established for the process $Z$, we know that the probability of exiting the domain $\cDc$ within any positive time $T$ tends to zero. The following result demonstrates that there exists a time $T$ such that this convergence is, in fact, no faster than exponential, with a rate determined by the height of the quasipotential $H$.

\begin{lemma}[{\cite[Lemma 5.7.18]{DZ}}]
\label{lm:tau(Z)<T_lower_bound}
    Under Assumptions~\ref{assu:1}--\ref{assu:4}, for any $\eta > 0$ and any $\rho > 0$ small enough, there exists a positive time $T$ such that
    \begin{equation*}
        \liminf_{\eps \to 0} \frac{\eps}{2} \log \inf_{x \in B_{\rho}(a)} \Prob (\tau(Z) \leq T) > -(H + \eta).
    \end{equation*}
\end{lemma}

The next lemma shows that it is highly unlikely for $Z$ to remain inside the domain $\cDc$ for a long period without reaching either $\partial \cDc$ or a small neighborhood around the attractor $a$. Recall the definition of $\tau^\prime_\rho$: for any $\rho > 0$ such that $B_\rho(a) \subset \cDc$, we define $\tau^\prime_\rho(Z) := \inf\{t \geq 0 : Z_t \in B_\rho(a) \cup \partial \cDc\}$, which represents the first time the process either approaches $a$ or exits $\cDc$.

\begin{lemma}[{\cite[Lemma 5.7.19]{DZ}}]
\label{lm:tau_prim_rho(Z)>t}
    Under Assumptions~\ref{assu:1}--\ref{assu:4}, we have
    \begin{equation*}
        \lim_{t \to \infty} \limsup_{\eps \to 0} \frac{\eps}{2} \log \sup_{x \in \cDc} \Prob (\tau_\rho^\prime(Z) > t) = -\infty.
    \end{equation*}
\end{lemma}



The following result shows that the probability that the process $Z$ starts around the point $a$ and touches some closed subset of $\cDc$ in some finite time $T$ is exponentially small.

\begin{lemma}
\label{lm:Z_in_Phi}
    Let Assumptions~\ref{assu:1}--\ref{assu:4} hold. For any positive time $T$ and for any compact $C \subset \R^d$ such that $\ds \inf_{y \in C} Q(y) < \infty$ define $\Phi := \{f \in C[0, T]: f(0) \in S_{2\rho}(a) \text{ and } \exists t \leq T \text{ such that } f(t) \in C\}$. Then, we have 
    \begin{equation*}
        \limsup_{\eps \to 0} \frac{\eps}{2} \log \sup_{x \in S_{2\rho}(a)} \Prob(Z \in \Phi) \leq - \inf_{y \in C} Q(y) + \sup_{z \in S_{2\rho}(a)} Q(z).
    \end{equation*}
\end{lemma}

\begin{proof}
    The proof consists in noting that, first, by Lemma~\ref{lm:Z_LDP_compact_starting_set}, we have:
    \begin{equation*}
        \limsup_{\eps \to 0} \frac{\eps}{2} \log \sup_{x \in S_{2\rho}(a)} \Prob(Z \in \Phi) \leq \inf_{y \in S_{2\rho}(a)} \inf_{f \in \Phi} I^y_T(f).
    \end{equation*}
    
    Second, we use the definition of the function $U$, and get 
    \begin{equation*}
        \inf_{y \in S_{2\rho(a)}}\inf_{f \in \Phi} I^y_T(f) \geq \inf_{z \in S_{2\rho}(a)} \inf_{y \in C} U(z, y).
    \end{equation*}
    By triangular inequality, we then have:
    \begin{equation*}
        \inf_{z \in S_{2\rho}(a)} \inf_{y \in C} U(z, y) \geq \inf_{y \in C} Q(y) - \sup_{z \in S_{2\rho}(a)} Q(z).
    \end{equation*}
\end{proof}

The next result shows that the process $Z$ does not move far from its initial position $x$ within a small time interval. Furthermore, for any specified level of proximity $r$ to the starting point $x$, and any desired exponential rate $c$, it is possible to choose a sufficiently small time $T(r, c)$ such that the probability of the event $\ds \{\sup_{t \leq T(r, c)} |Z_t - x| > r\}$ decreases to zero exponentially fast as $\eps \to 0$, with the rate $c$ specified in advance.

\begin{lemma}[{\cite[Lemma 5.7.23]{DZ}}]
\label{lm:|Z_t-x|>r}
    Under Assumptions~\ref{assu:1}--\ref{assu:4}, for any $r > 0$ and for any $c > 0$, there exists a positive time $T(r, c)$ such that
    \begin{equation*}
        \limsup_{\eps \to 0} \frac{\eps}{2} \log \sup_{x \in \cDc} \Prob \left(\sup_{t \leq T(r, c)} |Z_t - x| \geq r \right) < -c.
    \end{equation*}
\end{lemma}

\begin{small}
\bibliographystyle{plain}
\def\cprime{$'$}

\end{small}

\end{document}